\documentclass[reqno]{amsart}
\usepackage[utf8]{inputenc}
\usepackage[english]{babel}

\usepackage{amsfonts,amsthm,amssymb,mathtools}
\usepackage[alphabetic,abbrev]{amsrefs}
\numberwithin{equation}{section}

\usepackage[unicode]{hyperref}
\usepackage[capitalise,noabbrev]{cleveref} 
\crefname{equation}{}{}
\Crefname{equation}{}{}

\newtheorem{theorem}{Theorem}[section]
\newtheorem{lemma}[theorem]{Lemma}
\newtheorem{proposition}[theorem]{Proposition}

\theoremstyle{definition}
\newtheorem*{definition}{Definition}

\newtheorem{assumption}{Assumption}

\theoremstyle{remark}
\newtheorem{remark}[theorem]{Remark}
\newtheorem{example}[theorem]{Example}

\newcommand{\N}{\mathbb{N}}
\newcommand{\Z}{\mathbb{Z}}

\newcommand{\R}{\mathbb{R}}
\newcommand{\C}{\mathbb{C}}

\newcommand{\g}{\mathfrak{g}}

\DeclarePairedDelimiter\norm{\lVert}{\rVert}

\DeclareMathOperator{\supp}{supp}
\DeclareMathOperator{\id}{id}

\newcommand{\sloc}{\mathrm{sloc}}
\newcommand{\eigvp}[2]{\lambda_{\mathbf{#1}}^{#2}}

\newcommand{\ST}[1]{p_{\mathrm{ST}}(#1)}

\AtBeginDocument{%
   \def\MR#1{}
}

\begin{document}

\title[Spectral multipliers on two-step stratified Lie groups]{Spectral multipliers on two-step stratified Lie groups with degenerate group structure}
\author{Lars Niedorf}
\address{Department of Mathematics, University of Wisconsin-Madison,
480 Lincoln Dr, Madison, WI-53706, USA}
\email{niedorf@wisc.edu}
\date{\today}
\thanks{The author gratefully acknowledges the support by the Deutsche Forschungsgemeinschaft (DFG) through grant MU 761/12-1.}

\begin{abstract}
Let $L$ be a sub-Laplacian on a two-step stratified Lie group $G$ of topological dimension $d$. We prove new $L^p$-spectral multiplier estimates under the sharp regularity condition $s>d\left|1/p-1/2\right|$ in settings where the group structure of $G$ is degenerate, extending previously known results for the non-degenerate case. Our results include variants of the free two-step nilpotent group on three generators and Heisenberg--Reiter groups. The proof combines restriction type estimates with a detailed analysis of the sub-Riemannian geometry of $G$. A key novelty of our approach is the use of a refined spectral decomposition into caps on the unit sphere in the center of the Lie group.
\end{abstract}


\maketitle

\section{Introduction}

Let $\Delta$ be the Laplacian on $\R^d$. Given a Borel measurable function $F:\R\to\C$, called \textit{spectral multiplier}, the operator $F(-\Delta)$ on $L^2(\R^n)$ defined by functional calculus corresponds to a radial Fourier multiplier operator via
\[
\smash{(F(-\Delta)f)^\wedge(\xi) = F(|\xi|^2)\hat f(\xi)},\quad f\in\mathcal S(\R^n).
\]
The Mikhlin--Hörmander theorem \cite{Hoe60} implies that $F(-\Delta)$ extends to a bounded operator on $L^p(\R^n)$ for all $p\in (1,\infty)$ whenever
\[
\|F\|_{L^2_{s,\sloc}} = \sup_{t>0} {\norm{F(t\,\cdot\,)\eta }_{L^2_s(\R^n)}} < \infty\quad \text{for some } s>n/2,
\]
where $\eta\in C_c^\infty(\R^n\setminus\{0\})$ is a generic nonzero cutoff function, and $L^2_s(\R^n)\subseteq L^2(\R^n)$ is the Sobolev space of fractional order $s\ge 0$. The regularity condition $s>n/2$ is sharp and cannot be decreased \cite{SiWr01}.
If we ask only for boundedness for fixed $p$, it is now known (see \cite{Fe73,Ch85,Se86}) that for $1< p \le 2(d+1)/(d+3)$, the operator $F(-\Delta)$ is bounded on $L^p(\R^d)$ under the weaker assumption that
\[
\|F\|_{L^2_{s,\sloc}} < \infty \quad\text{for some } s > d\left|1/p-1/2\right|.
\]
We call such results \textit{$p$-specific} spectral multiplier estimates. These estimates, in turn, imply $L^p$-boundedness of Bochner--Riesz means $(1+\Delta)_+^\delta$ under the sharp order of regularity $\delta>d\left|1/p-1/2\right|-1/2$.

Theorems of Mikhlin--Hörmander type are known in various settings. Let
\begin{equation}\label{eq:sub-Laplacian}
L = - \left(X_1^2 +\dots+X_{d_1}^2 \right)
\end{equation}
be a left-invariant sub-Laplacian on a stratified Lie group $G$, where $X_1,\dots,X_{d_1}$ are left-invariant vector fields corresponding to a basis of the first layer of a given stratification of $G$ (for further details, see \cref{sec:basics,sec:restriction}). Let $G$ be equipped with a left-invariant Haar measure, and let $L^p(G)$ be the associated Lebesgue space. The sub-Laplacian $L$ extends to a positive, self-adjoint operator on $L^2(G)$. Thus, for any spectral multiplier $F : \mathbb{R} \to \mathbb{C}$, the operator $F(L)$ on $L^2(G)$ can be defined by functional calculus, which is bounded on $L^2(G)$ if and only if $F\in L^\infty$.

For those sub-Laplacians, Christ \cite{Ch91} and Mauceri--Meda \cite{MaMe90} showed the sufficiency of the threshold $s>Q/2$, where $Q$ is the so-called \textit{homogeneous dimension} of the underlying Lie group, which is in general strictly greater than the topological dimension $d$.

However, the threshold $s>Q/2$ in \cite{Ch91,MaMe90} is not sharp in general. It was improved to $s>d/2$ by Müller--Stein \cite{MueSt94} and Hebisch \cite{He93} for sub-Laplacians on the subclass of Heisenberg (-type) groups. Recent extensions of this result include several subclasses within the class of two-step stratified Lie groups \cite{MaMue14N,Ma15}, Grushin type operators \cite{DaMa20,DaMa22b}, and sub-Laplacians on classes of compact manifolds \cite{CoKlSi11,AhCoMaMue20}. The threshold $s>d/2$ in all these results is sharp \cite{MaMueGo23}. For general stratified Lie groups it is not known whether the threshold $s>d/2$ is always sufficient, and even the two-step case is far from being completely understood.

In sub-elliptic situations, even less is known about the sufficiency of the $p$-dependent threshold $\smash{s> d\left|1/p-1/2\right|}$. A first partial result for Grushin operators was given by Chen and Ouhabaz \cite{ChOu16}, which was later completed in \cite{Ni22}. The method of \cite{Ni22} was subsequently adapted to sub-Laplacians on Heisenberg-type groups \cite{Ni24} and further extended to the broader class of Métivier groups \cite{Ni25M}, a subclass of two-step stratified Lie groups characterized by a specific non-degeneracy condition. However, given that Mikhlin--Hörmander type theorems with sharp regularity condition $s>d/2$ are available for subclasses of two-step stratified Lie groups even where this non-degeneracy condition fails \cite{MaMue13,MaMue14N,Ma15}, the question arises whether corresponding $p$-specific spectral multiplier estimates also hold for these groups, or at least for certain subclasses.

This article aims to explore the case of sub-Laplacians $\smash{L = - (X_1^2 +\dots+X_{d_1}^2 )}$ within a particular subclass of two-step stratified Lie groups where the group structure may exhibit degeneracy.

Recall that $X_1,\dots,X_{d_1}$ is a basis of the first layer of a stratification of $G$. If we also fix a basis for the second layer of the chosen stratification, in the associated exponential coordinates, the group law of the Lie group $G$ is given by     
\[
(x,u) \cdot (x',u')= \left(x+x',u+u'+\tfrac 1 2 \, (\vec Jx)^\intercal x'\right) ,\quad(x,u), (x',u') \in \R^{d_1} \times\R^{d_2},
\]
where $\smash{(\vec Jx)^\intercal x' = \sum_{k=1}^{d_2} (J_k x)^\intercal x' e_k \in \R^{d_2}}$, with $\{e_1,\dots,e_{d_2}\}$ denoting the standard basis in $\R^{d_2}$, and $J_1,\dots,J_{d_2}$ being some skew-symmetric matrices acting on $\R^{d_1}$. In these exponential coordinates, the left-invariant vector fields $X_j$ are given by
\[
X_j=\partial_{x_j} + \frac 1 2 \sum_{k=1}^{d_2} (J_k x)^\intercal e_j \, \partial_{u_k}.
\]

The most prominent example of a two-step stratified Lie-group is the \textit{Heisenberg group}, where $d_2=1$, $d_1=2n$ for some $n\in\N\setminus\{0\}$, and $J$ is given by
\[
J= \begin{pmatrix}
    0 & I_{n} \\
    - I_{n} & 0
\end{pmatrix}.
\]
The class of Heisenberg groups is contained in the larger classes of Heisenberg type and Métivier groups, which are defined as follows:

\begin{definition}
Let $J_\mu:=\mu_1 J_1+\dots+\mu_{d_2}J_{d_2}$ for $\mu\in\R^{d_2}$.
\begin{itemize}
    \item $G$ is called a \textit{Métivier group} if $J_\mu$ is invertible for all $\mu\in\R^{d_2}\setminus\{0\}$.
    \item $G$ is called a \textit{Heisenberg type group} if $J_\mu$ is orthogonal for $|\mu|=1$.
\end{itemize}
\end{definition}

The class of Heisenberg type groups is a strict subclass of the class of Métivier groups. By homogeneity, if $G$ is a Heisenberg type group, we have
\[
J_\mu^2 = - |\mu|^2 \,\mathrm{id}_{\R^{d_1}}\quad\text{for all } \mu\in \R^{d_2}.
\]
Passing to the larger class of Métivier groups means that we give up this rotation invariant structure. Then, the matrices $J_\mu$ are no longer orthogonal for $|\mu|=1$ and the spectral decomposition of $J_\mu$ into eigenspaces -- which depends on $\mu$ -- is much more complicated, see also \cite[Lemma~5]{MaMue14N}. The same applies to the spectral information of the sub-Laplacian $L$ which is much harder to control and much more challenging than in the case of Heisenberg type groups.

The class of Lie groups considered in this paper satisfies the following two assumptions, where we allow the group structure to be \textit{degenerate} in the sense that the matrices $J_\mu$ might have a non-trivial kernel. Recall that the matrices $J_\mu$ act on $\R^{d_1}$, which in exponential coordinates is the first layer of the stratification $\R^{d_1}\oplus \R^{d_2}$ of the Lie algebra of $G$.

\begin{assumption}\label{assumptionA}
There exist integers $r_1,\dots,r_N > 0$, and an orthogonal decomposition $\R^{d_1} = V_1 \oplus \dots \oplus V_N$ such that, if $P_1,\dots,P_N$ are the corresponding orthogonal projections, then, for all $\mu \in \R^{d_2}\setminus\{0\}$ and all $n \in \{1,\dots,N\}$,
\begin{enumerate}
    \item $J_\mu P_n = P_n J_\mu$,
    \item $\smash{J^2_\mu P_n}$ has rank $2r_n$ and exactly one non-zero eigenvalue.
\end{enumerate}
\end{assumption}

In other words, $\mathbb{R}^{d_1}$ decomposes into  $N$ blocks, and the matrices $J_\mu$ act independently on each block. Within each block, $J_\mu^2$ has exactly one non-zero eigenvalue. This condition was previously considered in \cite{Ma15}.

\begin{assumption}\label{assumptionB}
If $x,x'\in \ker J_{\mu^0}$ for some $\mu^0\in \R^{d_2}\setminus\{0\}$, then $(J x)^\intercal x'=0$.
\end{assumption}

In light of \cref{assumptionA}, this assumption is natural to consider. If \cref{assumptionA} holds and the matrices $J_\mu$ restricted to $V_n$ are have a kernel of dimension at most one, then \cref{assumptionB} is automatically satisfied. This follows from the skew-symmetry of $J_\mu$, which implies that $(J_\mu y)^\intercal y = 0$ for any $y \in \mathbb{R}^{d_1}$, and thus $(J_\mu x)^\intercal x' = 0$ for all $x, x' \in \ker J_\mu$ and all $\mu \in \mathbb{R}^{d_2}$ if $\dim \ker J_\mu \le 1$.

We shed some more light on Assumptions \ref{assumptionA} and \ref{assumptionB} in \cref{sec:basics}, where we compare it with the aforementioned Heisenberg type and Métivier groups, and give concrete examples of groups that satisfy this assumption. Our examples include a variant of the free two-step nilpotent Lie group on three generators and Heisenberg--Reiter groups, which were previously considered in \cite{MaMue13,Ma15}.

It should be emphasized that \cref{assumptionA} depends a priori on the choice of coordinates on $G$. However, it is actually possible to give a coordinate-independent formulation, see \cite{Ma15}. We proceed with our main result.

\begin{theorem}\label{thm:main}
Let $L=-(X_1^2+\dots+X_{d_1})$ be a sub-Laplacian on a two-step stratified Lie group $G\cong \R^d = \R^{d_1}\times \R^{d_2}$, where $X_1,\dots,X_{d_1}$ are the left-invariant vector fields associated with a basis of the first layer of a stratification of $G$.
Suppose that $G$ satisfies Assumptions \ref{assumptionA} and \ref{assumptionB} and that additionally
\begin{equation}\label{eq:cond-main}
\bar d_1 := 2\left(r_1+\dots+r_N\right)\ge d_2-1
\end{equation}
in \cref{assumptionA}. Let
\[
p_{\bar d_1,d_2} := \frac{\bar d_1(d_2-1) +2d_2(d_2+1)}{\bar d_1(d_2-1) + 3d_2^2+1}.
\]
Then, if $1\le p \le p_{d_1,d_2}$, the following statements hold:
\begin{enumerate}
\item \label{main-(1)}%
If $p>1$ and if $F:\R\to \C$ is a bounded Borel function such that
\[
\|F\|_{L^2_{s,\sloc}} < \infty
\quad\text{for some } s > d \left(1/p - 1/2\right),
\]
then the operator $F(L)$ is bounded on $L^p(G)$, and
\[
\norm{F(L)}_{L^p\to L^p} \le C_{p,s} \|F\|_{L^2_{s,\sloc}}.
\]
\item \label{main-(2)}%
For any $\delta> d \left(1/p - 1/2\right)- 1/2$, the Bochner--Riesz means $(1-tL)^\delta_+$, $t\ge 0$, are uniformly bounded on $L^p(G)$.
\end{enumerate}
\end{theorem}

As shown in \cite{MaMueGo23}, the order of regularity $s > d \left( 1/p - 1/2 \right)$ is sharp. The threshold $\delta > d \left( 1/p - 1/2 \right) - 1/2$ corresponds to the same order of regularity as in the Bochner--Riesz conjecture, see \cite{Ta99}. It is important to emphasize that \cref{thm:main} should be interpreted as a result for high-dimensional settings, where $d_1$ is much larger than $d_2$. Specifically, \cref{assumptionA} implies that the function $\mu \mapsto \dim \ker J_\mu$ is constant on $\mathbb{R}^{d_2} \setminus \{0\}$. Therefore, if $r_0 = \dim \ker J_\mu$, then $\bar{d}_1 = d_1 - r_0$. Even with both Assumptions \ref{assumptionA} and \ref{assumptionB}, the condition \cref{eq:cond-main} is not necessarily satisfied. A counterexample can be found in the Heisenberg--Reiter groups discussed below, see \cref{ex:HR}.

If \cref{eq:cond-main} holds, then $\smash{p_{\bar{d}_1, d_2} \ge 1}$, and \cref{thm:main} gives results. In fact, note that $\smash{p_{\bar d_1,d_2}}$ is an increasing function in $\bar d_1$. Since $\smash{p_{\bar d_1,d_2}=1}$ in the limiting case $\smash{\bar d_1=d_2-1}$, the first part of \cref{thm:main} provides results whenever $\smash{\bar d_1>d_2-1}$, while the second part provides results for $\smash{\bar d_1\ge d_2-1}$.

Theorem \ref{thm:main} is completely new in the non-Métivier case. However, \cref{thm:main} also recovers some of the results of \cite{Ni25M} for Métivier groups, although the results of \cite{Ni25M} are stronger in this case. Note that \cref{assumptionB} is trivially satisfied for Métivier groups, although not every Métivier group necessarily satisfies \cref{assumptionA}, see \cite[Proposition 25]{MaMue14N}.

\subsection*{Structure of the paper}

In \cref{sec:basics}, we analyze the spectral properties of the matrices $J_\mu$ and explore certain classes of groups where the group structure is degenerate. In \cref{sec:restriction}, we introduce a spectral decomposition into caps on the unit sphere within the center of the Lie group and prove a restriction type estimate for the corresponding operators. In \cref{sec:reduction}, we reduce the proof of \cref{thm:main} to spectral multipliers whose frequencies are supported on dyadic scales.

In \cref{sec:1-layer} and \cref{sec:2-layer}, we analyze the supports of the convolution kernels of the dyadic spectral multipliers from \cref{sec:reduction} using Plancherel type estimates. The size of these supports plays a central role in the proof of \cref{thm:main}, which is completed in \cref{sec:proof}.

\subsection*{Acknowledgments}

I would like to thank Alessio Martini for insightful discussions on the subject of this paper and for his valuable comments on the proof of the main results.

\subsection*{Notation}

We let $\N=\{0,1,2,\dots\}$ and $\R^+=(0,\infty)$. Let $\mathcal S(G)$ denote the space of Schwartz functions on $G\cong\R^d$. The indicator function of a subset $A$ of some measurable space is be denoted by $\mathbf{1}_A$. Given two suitable functions $f_1,f_2$ on $G$, let $f*g$ be the group convolution given by
\[
f*g(x,u) = \int_{G} f(x',u')\,g\big((x',u')^{-1}(x,u)\big) \,d(x',u'),\quad (x,u)\in G,
\]
where $d(x',u')$ denotes the Lebesgue measure on $G$. For a function $f\in L^1(\R^n)$, the Fourier transform $\hat f$ is given by
\[
\hat f(\xi) = \int_{\R^n} f(x) e^{-i\xi x}\, dx,\quad
\xi\in\R^n.
\]
We denote the inverse Fourier transform by $\check f$. We write $A\lesssim B$ if $A\le C B$ for some constant $C$. If $A\lesssim B$ and $B\lesssim A$, we also write $A\sim B$.

\section{Preliminaries}\label{sec:basics}

\subsection{Two-step stratified Lie groups}

Let $G$ be a two-step stratified Lie group. By definition, this means that $G$ is a connected, simply connected nilpotent Lie group whose Lie algebra $\g$ admits a decomposition $\g=\g_1\oplus\g_2$ into two non-trivial subspaces $\g_1,\g_2\subseteq\g$, called layers, such that
\[
[\g_1,\g_1]=\g_2 \quad\text{and}\quad [\g,\g_2]=0.
\]
We refer to $\g_1$ and $\g_2$ as the \textit{first} and \textit{second} layer of $\g$, respectively. Let
\[
d_1=\dim \g_1\ge 1,\quad d_2=\dim \g_2\ge 1\quad \text{and}\quad d=\dim \g.
\]
In exponential coordinates, the group law of $G$ is given by
\[
(x,u)(x',u') = (x+x',u+u'+\tfrac 1 2 [x,x']),\quad x,x'\in\g_1,u,u'\in\g_2.
\]
We recall some prominent classes of two-step stratified Lie groups. Let $\langle \cdot,\cdot\rangle$ be an inner product with respect to which the layers $\g_1$ and $\g_2$ are orthogonal. Let $\g_2^*$ denote the dual of $\g_2$. For any $\mu\in\g_2^*$, let $\omega_\mu:\g_1\times\g_1\to\R$ be the skew-symmetric bilinear form given by
\[
\omega_\mu(x,x') = \mu([x,x']),\quad x,x'\in\g_1.
\]
The group $G$ is called a \textit{Métivier group} if $\omega_\mu$ is non-degenerate for all $\mu\in \g_2^*\setminus\{0\}$, which means that $\mathfrak r_\mu=\{0\}$, where $\mathfrak r_\mu$ is the radical
\[
\mathfrak r_\mu = \{x\in \g_1:\omega_\mu(x,x')=0\text{ for all } x'\in\g_1\}.
\]
The inner product $\langle\cdot,\cdot\rangle$ induces a norm on the dual $\g_2^*$ which we denote by $|\cdot|$. Let $J_\mu$ be the skew-symmetric endomorphism such that
\begin{equation}\label{eq:skew-form-ii}
\omega_\mu(x,x') =\langle J_\mu x,x'\rangle,\quad x,x'\in\g_1.
\end{equation}
Given a basis $X_1,\dots,X_{d_1}$ of the first layer $\g_1$ and a basis $U_1,\dots,U_{d_2}$ of the second layer $\g_2$, the group law in the associated coordinates is given by
\[
(x,u) \cdot (x',u')= \left(x+x',u+u'+\tfrac 1 2 \, (\vec Jx)^\intercal x'\right) ,\quad(x,u), (x',u') \in \R^{d_1} \times\R^{d_2},
\]
where $\smash{(\vec Jx)^\intercal x' = \sum_{k=1}^{d_2} (J_k x)^\intercal x' e_k \in \R^{d_2}}$ with $\{e_1,\dots,e_{d_2}\}$ being the standard basis in $\R^{d_2}$ 
and $J_k$ is the skew-symmetric matrix acting on $\R^{d_1}$ which corresponds to the endomorphism $\smash{J_{U_k^*}}$ under the chosen basis, with $U_k^*\in\g_2^*$ being the dual element to $U_k\in\g_2$. Moreover, if we identify we matrix $J_\mu$ on $\g_1$ with its representing matrix in the coordinates $X_1,\dots,X_{d_1}$ and write $\mu=\mu_1 U_1+\dots+\mu_{d_2} U_{d_2}$, then
\[
J_\mu=\mu_1 J_1 +\dots+\mu_{d_2} J_{d_2}.
\]
Henceforth, we will usually identify $\g_1 \cong \R^{d_1}$ and $\g_2 \cong \R^{d_2}$ by means of the bases $X_1,\dots,X_{d_1}$ and $U_1,\dots,U_{d_2}$.

Note that the group $G$ is a Métivier group if and only if $J_\mu$ is invertible for all $\mu\in\g_2^*\setminus\{0\}$. It is called a \textit{Heisenberg type group} if the endomorphisms $J_\mu$ are additionally orthogonal for all $\mu\in \g_2^*$ of length 1, which is equivalent to
\[
J_\mu^2 = - |\mu|^2 \id_{\g_1}\quad\text{for all } \mu\in \g_2^*.
\]
The class of Heisenberg type groups is a strict subclass of that of Métivier groups, an example of a Métivier group that is not of Heisenberg type can be found for instance in \cite{MueSe04}. Note that the dimension $d_1$ of the second layer of always even if $G$ is a Métivier group. If $d_1$ is even, $d_2=1$ and
\[
J_\mu = \mu \begin{pmatrix}
0 & I_{d_1/2} \\
-I_{d_1/2} & 0 \\
\end{pmatrix} ,\quad\mu\in\R,
\]
then $G$ is a \textit{Heisenberg group}. 

\subsection{\texorpdfstring{Spectral analysis of the matrices $J_\mu$}{Spectral analysis of the matrices Jmu}}

Next, we analyze the spectral properties of the family of matrices $J_\mu$. Recall that any matrix $J_\mu$ is skew-symmetric, so the matrix $iJ_\mu$ on the complexification $\g_1^\C\cong \C^{d_1}$ is self-adjoint, which implies that $-J_\mu^2$ is also self-adjoint.

The matrices $J_\mu$ admit a “simultaneous spectral decomposition” if $\mu$ lies in a Zariski open subset $\g_{2,r}^*$ of $\g_2^*$. This is made precise in the following statement, which can be found, for example, in \cite[Lemma 5]{MaMue14N} and \cite[Proposition~3.4]{Ni25M}.

\begin{proposition}\label{prop:J_mu-spectral}
There exist
\begin{itemize}
    \item a non-empty, homogeneous Zariski-open subset $\g_{2,r}^*$ of $\g_2^*$,
    \item numbers $N\in\N\setminus\{0\}$, $r_0\in\N$, $\mathbf r=(r_1,\dots,r_N)\in\smash{(\N\setminus\{0\})^N}$,
    \item a function $\mu \mapsto \mathbf{b}^\mu = (b^\mu_1,\dots,b^\mu_N)\in [0,\infty)^N$ on $\g_2^*$, and
    \item functions $\mu\mapsto P_n^\mu $ on $\g_{2,r}^*$ with $P_n^\mu:\g_1\to\g_1$, $n\in\{1,\dots,N\}$,
\end{itemize}
such that
\begin{equation}\label{eq:J_mu-spectral}
-J_\mu^2 = \sum_{n=1}^N (b_n^\mu)^2 P_n^\mu \quad \text{for all } \mu \in \g_{2,r}^*.
\end{equation}
Moreover, the functions $\mu \mapsto b^\mu_n$ and $\mu\mapsto P_n^\mu$ satisfy the following two properties:
\begin{enumerate}
\item[\textup{(i)}] the functions $\mu \mapsto b^\mu_n$ are homogeneous of degree~$1$ and continuous on $\g_2^*$, real analytic on $\g_{2,r}^*$, and satisfy $b_n^\mu > 0$ for all $\mu \in \g_{2,r}^*$ and $n \in \{1,\dots,N\}$, and $b_n^\mu \neq b_{n'}^\mu$ if $n \neq n'$ for all $\mu \in \g_{2,r}^*$ and $n,n' \in \{1,\dots,N\}$,
\item[\textup{(ii)}] the functions $\mu\mapsto P_n^\mu$ are (componentwise) real analytic on $\g_{2,r}^*$, homogeneous of degree $0$, and the maps $P_n^\mu$ are orthogonal projections on $\g_1$ of rank $2r_n$ for all $\mu \in \g_{2,r}^*$, with pairwise orthogonal ranges.
\end{enumerate}
\end{proposition}

A few comments about \cref{prop:J_mu-spectral} are in order. In the special case where $G$ is a Métivier group, we have $r_0=0$ since $\ker J_\mu=0$ for $\mu\neq 0$. If $G$ is even a Heisenberg type group, we may choose $\g_{2,r}^*=\g_2^*\setminus\{0\}$, $N=1$, $\mathbf{b}^\mu=|\mu|$, and the only projection is identity map on $\g_1$. However, for a general two-step stratified Lie group, it may happen in the above decomposition that $b_n^\mu = 0$ or $b_n^\mu = b_{n'}^\mu$ for $n\neq n'$ if $\mu$ is in the Zariski closed set $\g_2^*\setminus\g_{2,r}^*$. This structure is generally difficult to handle, since the functions $\mu \mapsto b^\mu_n$ and $\mu\mapsto P_n^\mu$ may have singularities on the set $\g_2^*\setminus\g_{2,r}^*$. For more details on these singularities, we refer the reader to \cite{MaMue14N}, and also \cite{MaMue16}.

\subsection{Particular classes of non-Métiver groups}

We shed more light on Assumptions \ref{assumptionA} and \ref{assumptionB}, which define the class of Lie groups considered in this paper.

In \cref{assumptionA}, we require a decomposition $\g_1 = V_1 \oplus \dots \oplus V_N$ such that the matrices $J_\mu$ act independently on each subspace $V_n$. Moreover, $J_\mu^2$ must admit a unique non-zero eigenvalue. Let $P_n$ denote the orthogonal projection onto $V_n$. Under \cref{assumptionA}, the spectral decomposition of the matrix $J_\mu^2 P_n$, when restricted to $V_n$, consists of a single projection, which is the projection onto the unique eigenspace of $J_\mu^2 P_n$. Since we assume that the rank of $J_\mu^2 P_n$ remains constant, the rank of this projection also remains constant as $\mu$ varies over $\g_2^* \setminus \{0\}$.

Comparing this setup with the decomposition in \cref{prop:J_mu-spectral}, satisfying \cref{assumptionA} means that the decomposition \cref{eq:J_mu-spectral} extends actually to $\g_{2,r}^* = \g_2^* \setminus \{0\}$, that is,
\begin{equation}\label{eq:J_mu-spectral-1}
-J_\mu^2 = \sum_{n=1}^N (b_n^\mu)^2 P_n^\mu \quad \text{for all } \mu\in\g_2^* \setminus \{0\},
\end{equation}
where $b_n^\mu > 0$ for all $\mu \in \g_2^* \setminus \{0\}$. However, in \cref{eq:J_mu-spectral-1}, the projections $P_n^\mu$ are not necessarily the projections onto the full eigenspace associated with $b_n^\mu$, since it may happen that $b_n^\mu = b_{n'}^\mu$ for $n\neq n'$. Even though we require that there is a unique eigenvalue on each block $V_n$, the eigenvalues for some blocks may coincide. 

Assumption \ref{assumptionA} is the same as Assumption (A) in \cite{Ma15}. This assumption plays a crucial role in \cite{Ma15}, where a Mikhlin--Hörmander type result with sharp order of regularity $s>d/2$ is proved. The key to the proof of the main result of \cite{Ma15} is a weighted Plancherel estimate with a weight that depends only on the second layer~$\g_2$. Specifically, let $L = -(X_1^2 + \dots + X_{d_1}^2)$ denote the sub-Laplacian associated with the vector fields $X_1, \dots, X_{d_1}$, and let  
\[
U = \left(-(U_1^2 + \dots + U_{d_2}^2)\right)^{1/2} = \left(-\Delta_{\mathfrak{g}_2}\right)^{1/2}
\]  
be the square root of the negative partial Laplacian on the second layer $\mathfrak{g}_2$. Then, for all functions $F, \chi \in C_c^\infty(\mathbb{R}^+)$, the (group) convolution kernel $\mathcal{K}_\ell$ associated with the operator $F(L)\chi(2^\ell U)$ satisfies  
\begin{equation}\label{eq:Plancherel-2nd-layer-expl}
\int_G \left| |u|^\alpha \, \mathcal{K}_{\ell}(x,u) \right|^2 \, d(x,u) \lesssim 2^{\ell(2\alpha-d_2)} \|F\|_{L_2^\alpha}^2
\end{equation}
for all $\alpha \geq 0$ and $\ell \in \mathbb{Z}$. Here, the operator $F(L)\chi(2^\ell U)$ should be thought of as a dyadic piece of a decomposition of the operator $F(L)$.

By a partial Fourier transform, the weight $|u|^\alpha$ corresponds to taking derivatives of $b^\mu_n$ and $P_n^\mu$ with respect to the variable $\mu$. In fact, proving an estimate such as \cref{eq:Plancherel-2nd-layer-expl} beyond two-step stratified Lie groups where \cref{assumptionA} is violated and where the functions $\mu \mapsto b^\mu_n$ and $\mu\mapsto P_n^\mu$ may admit singularities is a highly non-trivial problem. Versions of such weighted Plancherel estimates are available under the assumption that one has good control over the derivatives of both $\mu \mapsto b^\mu_n$ and $\mu\mapsto P_n^\mu$, but beyond that not much is known about the validity of such estimates, see \cite{MaMue14N} for more details.

To prove the statement of \cref{thm:main}, \cref{assumptionA} alone does not appear sufficient, and we need to additionally require \cref{assumptionB}, which is satisfied by several notable classes of non-Métivier groups. In \cref{assumptionB}, we require that $\langle J_\mu x, x' \rangle = 0$ for all $\mu \in \mathfrak{g}_2^*$ and all $x, x' \in \mathfrak{g}_1$ that lie in the kernel of $J_{\mu^0}$ for some fixed $\mu^0 \in \mathfrak{g}_2^*$. This condition is automatically satisfied of the kernel of the matrix $J_\mu$ is one-dimensional on each subspace $V_n$. While many of the arguments in this paper remain valid without \cref{assumptionB}, this assumption seems to be vital for a crucial step in the proof.

\subsection{Some examples} Finally, we give two examples of classes of two-step stratified Lie groups for which the results of \cref{thm:main} hold.

\begin{example}[Copies of $N_{3,2}$ glued along their centers]\label{ex:N32}
The free $2$-step nilpotent Lie group $N_{3,2}$ on $3$ generators is the simply connected, connected nilpotent Lie group defined by the relations
\[[X_1,X_2] = Y_3, \quad [X_2,X_3] = Y_1, \quad [X_3,X_1] = Y_2,\]
where $X_j,Y_j$, $j\in\{1,2,3\}$ is a basis of its Lie algebra. Then, in exponential coordinates, $N_{3,2}$ can be identified with $\R^3_x \times \R^3_y$, where the group law is given by
\[(x,y) \cdot (x',y') = (x+x',y+y' + x\wedge x'/2)\]
and $x \wedge x'$ denotes the cross product of $x,x' \in \R^3$. This particular group was previously considered in \cite{MaMue13}. Now, for any $n\in\{1,\dots,N\}$, let $G_n=N_{3,2}$. Let $G$ be the the quotient of $G_1\times \dots\times G_2$ given by identifying the centers of the groups $G_n$ via the identity map, that is, $G$ is given by the relations 
\[[X_{n,1},X_{n,2}] = Y_3, \quad [X_{n,2},X_{n,3}] = Y_1, \quad [X_{n,3},X_{n,1}] = Y_2,\]
where $n\in\{1,\dots,N\}$ and $X_{n,j},Y_j$, $n\in\{1,\dots,N\}$, $j\in\{1,2,3\}$ is a basis of the Lie algebra of $G$. In particular, $G$ is a Lie group of topological dimension $3(N+1)$ whose center has dimension $d_2=3$.

We chose the usual Euclidean inner product on $G=\R^{3N}\times \R^3$. If $\smash{J_\mu^G}$ is defined as in \cref{eq:skew-form-ii}, then
\[
J_\mu^{G} = \begin{pmatrix}
J_\mu^{N_{3,2}} & & \\
 & \ddots & \\
 &  & J_\mu^{N_{3,2}}\\
\end{pmatrix},\quad\mu\in \R^3,
\]
where
\[
J_\mu^{N_{3,2}}=
\begin{pmatrix}
0 & -\mu_3 & \mu_2 \\
\mu_3 & 0 & -\mu_1 \\
-\mu_2 & \mu_1 & 0
\end{pmatrix},\quad\mu=(\mu_1,\mu_2,\mu_3).
\]
If we split $\R^{3N}=\R^3\times\dots\times\R^3$ and choose $P_n:\R^{3N}\to\R^{3N}$ as the canonical projection onto the $n$-th component, then $J_\mu P_n = P_n J_\mu$ for all $\mu\in\R^3$. Moreover, note that
\[
\left(J_\mu^{N_{3,2}}\right)^2 = 
\begin{pmatrix}
-\mu_2^2 - \mu_3^2 & \mu_1 \mu_2 & \mu_1 \mu_3\\
\mu_1 \mu_2 & -\mu_1^2 - \mu_3^2 & \mu_2 \mu_3\\
\mu_1 \mu_3 & \mu_2 \mu_3 & -\mu_1^2 - \mu_2^2
\end{pmatrix}.
\]
For $\mu \in \g_2^*\setminus\{0\}$, this matrix has a unique nonzero eigenvalue, which is $-|\mu|^2$ and which is associated with the eigenvectors $(-\mu_3,0,\mu_1)$ and $(-\mu_2,\mu_1,0)$, while its kernel is given by the subspace spanned by $\mu$. Thus, the constructed group satisfies both Assumptions \ref{assumptionA} and \ref{assumptionB}.

Finally, we check the dimensional conditions of \cref{thm:main}. If we let $2r_n$ be the rank of $J_\mu^2 P_n$, then $2\left(r_1+\dots+r_N\right)=2N\ge 2 = d_2-1$. Moreover,
\begin{align*}
p_{\bar d_1,d_2}
& = \frac{\bar d_1(d_2-1) +2d_2(d_2+1)}{\bar d_1(d_2-1) + 3d_2^2+1}\bigg|_{(\bar d_1,d_2)=(2N,3)} \\
& = 2 - \frac{8}{N+7}.
\end{align*}
In particular, the first part of \cref{thm:main} provides results only if $N\ge 2$, while the second part provides results for all $N\in\N\setminus\{0\}$ (but only for $p = 1$ when $N = 2$). 
\end{example}

\begin{example}[Heisenberg--Reiter groups]\label{ex:HR}
Given some integers $N>0$ and $d_2>0$, the Heisenberg--Reiter group $\mathbb{H}_{N,d_2}$ is defined by endowing $\R^{d_2 \times N} \times \R^N \times \R^{d_2}$ with the group law
\[
(x,y,u) \cdot (x',y',u') = (x+x',y+y',u+u'+(xy'-x'y)/2),
\]
where $x,x'\in \R^{d_2 \times N}$, $y,y'\in \R^N$, and $u,u'\in \R^{d_2}$.
Note that $\R^{d_2 \times N}$ is the set of the real $d_2 \times N$ matrices, and the products $xy',x'y$ above shall be interpreted in the sense of matrix multiplication. $\mathbb{H}_{N,d_2}$ is a two-step stratified Lie group of homogeneous dimension $Q = Nd_2 + N + 2d_2$ and topological dimension $d = Nd_2 + N + d_2$. This group was previously considered in \cite{Ma15}.

We equip $\mathbb{H}_{N,d_2}$ with the usual Euclidean inner product. If $\smash{J_\mu^{\mathbb{H}_{N,d_2}}}$ is defined as in \cref{eq:skew-form-ii}, then
\[
\langle J_\mu^{\mathbb{H}_{N,d_2}} (x,y),(x',y')\rangle
 = \langle \mu , x y' - x'y \rangle.
\]
Thus, for $\mu\neq 0$, the kernel of $\smash{J_\mu^{\mathbb{H}_{N,d_2}}}$ is given by $\mu^\perp \times \{0\}$, that is, $\smash{J_\mu^{\mathbb{H}_{N,d_2}}}(x,y)=0$ if and only if $\mu^\intercal x=0$ and $y=0$. So the dimension of 
the kernel is $(d_2-1)N$. In particular, $\mathbb{H}_{N,d_2}$ does not fall into the class of Métivier groups unless $d_2 = 1$, in which case it is the $(2N+1)$-dimensional Heisenberg group $\mathbb{H}_N$.

An orthogonal decomposition of the first layer is given by the isomorphism
\begin{equation}\label{eq:Heisenberg-Reiter-proj}
\R^{d_2 \times N} \times \R^N\cong (\R^{d_2} \times \R)^N.
\end{equation}
In other words, the Heisenberg--Reiter group $\mathbb{H}_{N,d_2}$ is isomorphic to $N$ copies of $\mathbb{H}_{1,d_2}$ with are glued along their centers. Thus, to verify that $\mathbb{H}_{N,d_2}$ satisfies Assumptions \ref{assumptionA} and \ref{assumptionB}, it suffices to verify this for each block $\mathbb{H}_{1,d_2}$.

For $(x,y),(x',y')\in  \R^{d_2}\times \R$, we have
\begin{align*}
\langle \mu , x y' - x'y \rangle 
& = x^\intercal \mu\,  y' - y\, \mu^\intercal x' \\
& = (x^\intercal \ y )
 \begin{pmatrix}
     0 & \mu \\
     -\mu^\intercal & 0 
 \end{pmatrix}
 \begin{pmatrix}
 x'\\y'   
\end{pmatrix}.
\end{align*}
Thus,
\[
J_\mu = J_\mu^{\mathbb{H}_{1,2}} = \begin{pmatrix}
     0 & \mu \\
     -\mu^\intercal & 0 
 \end{pmatrix}
 \quad\text{and}\quad
 J_\mu^2
 = 
 \begin{pmatrix}
 -\mu\,\mu^\intercal & 0 \\
0  & -|\mu|^2
\end{pmatrix}.
\]
Since $J_\mu$ is a skew-symmetric matrix of size $d_2 + 1$ and since the kernel of $J_\mu$ has dimension $d_2 - 1$, $\smash{J_\mu^2}$ must have a unique non-zero eigenvalue, whence $\smash{\mathbb{H}_{1,d_2}}$ satisfies \cref{assumptionA}. More explicitly, since the kernel of $J_\mu$ is given by $\smash{\mu^\perp\times \{0\}}$, if an element $(x,y)\in \R^{d_2}\times\R$ lies in the eigenspace of the non-zero eigenvalue, then $x=\lambda \mu$ for some $\lambda\in\R$. This implies that the non-zero eigenvalue is $-|\mu|^2$, which vanishes if and only if $\mu = 0$.

Moreover, if $(x,y),(x',y')\in \ker J_{\mu^0}$ for some $\mu^0\in\R^{d_2}\setminus\{0\}$, then $y=y'=0$ and thus $\langle \mu , x y' - x'y \rangle=0$ for all $\mu\in\R^{d_2}$, so \cref{assumptionB} is satisfied as well.

We check the remaining conditions of \cref{thm:main}. Let $P_n$ be the orthogonal projection associated with the decomposition \cref{eq:Heisenberg-Reiter-proj}. Then
\[
2r_n := \mathrm{rank}\big((J_\mu^{\mathbb{H}_{N,d_2}})^2 P_n\big) = \mathrm{rank}\big((J_\mu^{\mathbb{H}_{1,d_2}})^2\big)=1.
\]
Hence $2\left(r_1+\dots+r_N\right)=2N$. Thus, \cref{thm:main} only gives results for those Heisenberg--Reiter groups $\smash{\mathbb{H}_{N,d_2}}$ that satisfy $2N\ge d_2-1$.

\end{example}

\subsection{Projections onto the kernel} Let $P_0^\mu$ denote the orthogonal projection onto the kernel of $J_\mu$. Then, in view of \cref{eq:J_mu-spectral-1}, we have
\[
P_0^\mu = \id_{\g_1} - \left(P_1^\mu+\dots+P_N^\mu\right)\quad\text{for all } \mu\in\g_2^*\setminus\{0\},
\]
and $\mu\mapsto P_0^\mu$ is smooth on $\g_2^*\setminus\{0\}$. Thus, since $\mu\mapsto P_n^\mu$ is homogeneous, we have
\begin{equation}\label{eq:proj-radical}
\smash{\|P_0^\mu - P_0^{\mu'}\| \lesssim |\mu-\mu'| \quad\text{for all } \mu,\mu' \in \g_2^*\setminus\{0\}.}
\end{equation}
This estimate will play a role later in the proof of \cref{thm:main}. However, the validity of \cref{eq:proj-radical} does not depend on \cref{assumptionA} or \cref{assumptionB}. Instead, it is sufficient to require that the dimension of $\ker J_\mu$ remains constant as a function of $\mu$ on $\g_2^* \setminus \{0\}$, as will be shown in the next lemma.

\begin{lemma}\label{lem:proj-radical}
If $\mu\mapsto\dim \ker J_\mu$ is constant on $\g_2^*\setminus\{0\}$, then:
\begin{enumerate}
\item There is some $F\in C_c^\infty(\R)$ such that $F(iJ_\mu)|_{\g_1}=P_0^\mu$ for all $\g_2^*\setminus\{0\}$. 
\item The map $\mu\mapsto P_0^\mu$ is smooth and Lipschitz continuous on $\g_2^*\setminus\{0\}$, and
\[
\smash{\|P_0^\mu - P_0^{\mu'}\| \lesssim \|J_\mu-J_{\mu'}\| \quad\text{for all } \mu,\mu' \in \g_2^*\setminus\{0\}.}
\]
\end{enumerate}
\end{lemma}

\begin{proof}
As in \cref{prop:J_mu-spectral}, let $b_1^\mu,\dots,b_N^\mu$ be the positive eigenvalues of $iJ_\mu$. Then $b_n^\mu>0$ for all $\mu$ lying in the Zariski open subset $\g_{2,r}^*\subseteq \g_2^*$. Since $\mu\mapsto\dim \ker J_\mu$ is constant on $\g_2^*\setminus\{0\}$, we have even $b_n^\mu>0$ for all $\mu\in \g_2^*$. Thus, there are constants $C,c>0$ such that
\[
c < b_n^\mu < C \quad\text{for } |\mu|=1 \text{ and } n\in\{1,\dots,N\}.
\]
Note that $iJ_\mu$ is a self-adjoint endomorphism on $\g_1^\C=\C^{d_1}$. Let $F:\R\to \R$ be a smooth function with $\supp F\subseteq [-c,c]$ and $F(0)=1$. Then, via functional calculus, we have $F(iJ_\mu)|_{\g_1}=P_0^\mu$ for all $\mu\in\g_2^*\setminus\{0\}$, and the Fourier inversion formula yields
\[
P_0^\mu = \frac{1}{2\pi} \int_\R \hat F(\tau) e^{-\tau J_\mu} \,d\tau.
\]
Hence, $\mu\mapsto P_0^\mu$ is smooth on $\g_1$. Moreover,
\[
\|P_0^\mu - P_0^{\mu'}  \|
 \le \int_\R |\hat F(\tau)| \, \| e^{-\tau J_\mu}-e^{-\tau J_{\mu'}} \|\,d\tau.
\]
By the spectral theorem, we have $\smash{\|e^{iA}\|\lesssim \sup_{\lambda\in\R}|e^{i \lambda}| = 1}$ for all symmetric matrices $A$ on $\smash{\g_1^\C}$. Thus, by the mean value theorem,
\[
\| e^{-\tau J_\mu}-e^{-\tau J_{\mu'}} \|
\le \tau\left\| J_\mu- J_{\mu'} \right\|.
\]
Since $\hat F$ is a Schwartz function, we obtain $\|P_0^\mu - P_0^{\mu'}\| \lesssim \|J_\mu-J_{\mu'}\|$.
\end{proof}

\section{Truncated restriction type estimates on caps}\label{sec:restriction}

The proof of \cref{thm:main} relies on exploiting suitable restriction type estimates. Similar estimates are also used in \cite{Ni25M}. However, a novelty of the restriction type estimates used here is the additional spectral decomposition into caps on the unit sphere within the center of the Lie group. This additional decomposition is crucial for deriving spectral multiplier results for groups with degenerate group structures. The restriction type estimates stated here hold actually for general two-step stratified Lie group.

Let $G$ be a general two-step stratified Lie group. Let again $\g=\g_1\oplus\g_2$ be a stratification of its Lie algebra which is orthogonal with respect to a chosen inner product $\langle\cdot,\cdot\rangle$. As before, given an orthonormal basis $X_1,\dots,X_{d_1}$ of the first layer $\g_1$, we identify these by left-invariant vector fields via the Lie derivative, and consider the associated sub-Laplacian $L$, which is the second-order differential operator
\[
L = -(X_1^2+\dots+X_{d_1}^2).
\]
Let $U_1,\dots,U_{d_2}$ be an orthonormal basis of $\g_2$. The operators $L,-iU_1,\dots,-iU_{d_2}$ admit a joint functional calculus \cite{Ma11}, which means that we can define the operator $F(L,\mathbf{U})$ on $L^2(G)$ for any Borel function $F:\R \times \R^{d_2}\to\C$, where $\mathbf U$ is the vector of differential operators
$\mathbf U = (-iU_1,\dots,-iU_{d_2})$.

\subsection{A decomposition into caps}\label{sec:restriction-caps}

Using the joint functional calculus, we introduce an additional truncation along the spectrum of the operator
\[
U=(-(U_1^2+\dots+U_{d_2}^2))^{1/2},
\]
which corresponds to a partial Laplacian on the second layer $\g_2$. This allows us to write any operator $F(L)$ as
\[
F(L) = \sum_{\ell \in\Z} F(L) \chi(2^\ell U),
\]
where $\chi\in C_c^\infty(\R^+)$ and $(\chi(2^{\ell} \cdot))_{\ell\in\Z}$ forms a partition of unity. In fact, as we will see in \cref{prop:p=1} below, there exists a number $\ell_0 \in \mathbb{N}$ such that
\[
F(L)\chi(2^\ell U) = 0 \quad\text{for all } \ell < -\ell_0.
\]

The additional spectral decomposition into caps on the unit sphere $S^{d_2-1}\subseteq \g_2^*$ is as follows: Given $\delta>0$, let $(\zeta_j)_{j\in I}$ be a a smooth partition of unity subordinate to some open cover $(V_j)_{j\in I}$ of caps of size $\delta>0$ of the unit sphere $S^{d_2-1}$. Any function $\zeta_j:S^{d_2-1}\to\R$ extends to a smooth function $\zeta_j:\g_2^*\setminus\{0\}\to\R$ that is homogeneous of degree 0. This allows us to decompose the operator $F(L)$ as
\[
F(L) = \sum_{\ell \in\Z} \sum_{j\in I} F(L) \chi(2^\ell U) \zeta_j(\mathbf U).
\]
In the proof of \cref{thm:main}, the size of the caps will depend on $\ell$.

\subsection{Restriction type estimates}

The restriction type estimates of \cref{thm:restriction-type} below provide bounds for the operators $F(L) \chi(2^\ell U) \zeta_j(\mathbf U)$. These estimates play a crucial role in the proof of \cref{thm:main}.

For $M\in\R^+$, let $\|\cdot\|_{M,2}$ be the norm given by
\[
\|F\|_{M,2} = \bigg(\frac{1}{M} \sum_{K\in\Z} \,\sup _{\lambda \in [\frac{K-1}{M}, \frac{K}{M})}|F(\lambda)|^2\bigg)^{1 / 2}.
\]
For $n \in \mathbb{N} \setminus \{0\}$, let $\ST{n}$ denote the Stein--Tomas exponent on $\mathbb{R}^n$, which is  
\[
\ST{n} = \frac{2(n+1)}{n+3}.
\]  
Additionally, we set $\ST{d_1, d_2} = \min\{\ST{d_1}, \ST{d_2}\}$.

\begin{theorem}\label{thm:restriction-type}
Let $1\le p\le \ST{d_1,d_2}$. Let $A\subseteq \R^+$ be a compact subset, and let $\delta>0$ and $(\zeta_j)_{j\in I}$ be defined as above. If $F:\R\to\C$ is a bounded Borel function supported in $A$ and $\chi\in C_c^\infty(\R^+)$, then
\begin{equation}\label{eq:restriction-type}
\norm{ F(L)\chi(2^\ell U) \zeta_j(\mathbf U) }_{p\to 2}
\le C_{A,p,\chi} \, \delta^{\frac{d_2-1} 2 (1-\vartheta_p)} \,2^{-\ell d_2(\frac 1 p - \frac 1 2)}  \, \|F\|_2^{1-\theta_p} \|F\|_{2^{\ell},2}^{\theta_p}
\end{equation}
for all $\ell\in\Z$, where $\vartheta_p,\theta_p\in [0,1]$ satisfy
\[
1/p = (1-\vartheta_p) + \vartheta_p/\ST{d_2}\quad\text{and}\quad 1/p = (1-\theta_p) + \theta_p/\ST{d_1,d_2}.
\]
\end{theorem}

Note that Theorem \ref{thm:restriction-type} holds for general two-step stratified Lie groups. However, the proof of \cref{thm:main} requires restricting to a subclass of stratified Lie groups that satisfies both \cref{assumptionA} and \cref{assumptionB}. This restriction arises because, in particular, we rely on Plancherel type estimates, which unfortunately do not hold in general for arbitrary two-step stratified Lie groups.

As indicated by our notation, \cref{eq:restriction-type} is derived through interpolation between the endpoints $p = 1$, $p = \ST{d_1, d_2}$, and $p = \ST{d_2}$. Under the assumptions of \cref{thm:main}, which require $\bar{d}_1 \geq d_2 - 1$, we observe that $\ST{d_2} \leq \ST{d_1}$. Therefore, in the proof of \cref{thm:main}, we apply \cref{thm:restriction-type} with $\ST{d_2} = \ST{d_1, d_2}$ and $\vartheta_p = \theta_p$. In that case, it is possible to derive \cref{eq:restriction-type} via interpolation, combining the estimate for $p = 1$ from \cref{prop:p=1} below and the estimate from \cite{Ni25R} for $p = \ST{d_2}$.

Moreover, note that, by (3.19) and (3.29) of \cite{CoSi01}, we have the Sobolev type embedding
\begin{equation}\label{eq:sobolev-embedding}
\|F\|_{L^2}\le \|F\|_{M,2} \le C_s \big( \|F\|_{L^2} + M^{-s} \|F\|_{L^2_s} \big)\quad\text{for all } s> 1/2.
\end{equation}
In the special case where $G$ is a Heisenberg type group, it is possible to prove a restriction type estimate with $\|F\|_2$ on the right hand side of \cref{eq:restriction-type}, see \cite[Theorem 3.2]{Ni24}. With the $L^2$-norm on the right hand side, restriction type estimates are equivalent to restriction type estimates for the Strichartz projectors $\mathcal P_\lambda$, which are formally given by $\mathcal P_\lambda=\delta_\lambda(L)$, where $\delta_\lambda$ is the Dirac delta distribution at $\lambda\in\R$, see \cite[Proposition 4.1]{SiYaYa14} and also \cite{LiZh11}.

\subsection{\texorpdfstring{Proof of \cref{thm:restriction-type}}{Proof of Theorem 3.1}}

For $f\in L^1(G)$ and $\mu\in\g_2^*$, let $f^\mu$ denote the $\mu$-section of the partial Fourier transform along the second layer $\g_2$ given by
\begin{equation}\label{eq:partial-FT}
f^\mu(x) = \int_{\g_2} f(x,u) e^{- i \langle\mu, u\rangle} \, du,\quad x\in \g_1.
\end{equation}
Up to some constant, this defines an isometry $\mathcal F_2:L^2(\g_1\times \g_2)\to L^2(\g_1\times \g_2^*)$. Given $f\in L^2(G)$, we also write $f^\mu=(\mathcal F_2 f)(\cdot,\mu)$ (for almost all $\mu\in\g_2^*$) in the following. Moreover, let $L^\mu$ be the \textit{$\mu$-twisted Laplacian} on $\g_1$, which is the second order differential operator on $\g_1$ given by
\[
L^\mu = -\Delta_x + \tfrac 1 4 |J_\mu x|^2 - i\, ( J_\mu x)^\intercal \, \nabla ,
\]
where $( J_\mu x)^\intercal \, \nabla = \sum_{j=1}^{d_1}(J_\mu x)_j \, \partial_{x_j}$, with $(J_\mu x)_j$ denoting the $j$-th entry of $J_\mu x$.

The operator $L^\mu$ is self-adjoint, and as such, it admits a functional calculus. Moreover, the functional calculi of $L$, $\mathbf{U}$, and $L^\mu$ are compatible via the partial Fourier transform, which follows from \cite[Proposition 1.1]{Mue90}, see also \cite[Proposition 3.1]{Ni25R}.

\begin{lemma}\label{prop:joint-calculus}
If $F:\R\times\R^{d_2}\to\C$ is a bounded Borel function, then
\[
\left(F(L,\mathbf U)f\right)^\mu = F(L^\mu,\mu) f^\mu
\]
for all $f\in L^2(G)$ and almost all $\mu\in\g_2^*$.
\end{lemma}

By \cref{prop:joint-calculus}, we have
\begin{equation}\label{eq:joint-calculus-ii}
\big(F(L)\chi(2^\ell U)\zeta_j(\mathbf U)f\big)^\mu = F(L^\mu)\chi(2^\ell |\mu|) \zeta_j(\mu) f^\mu
\end{equation}
for all $\mu\in\g_2^*$ and all $f\in\mathcal S(G)$. To prove \cref{thm:main}, we apply the $L^p$-$L^2$ spectral cluster estimate from \cite[Theorem 4.1]{Ni25R} for the twisted Laplacian $L^\mu$.

\begin{proposition}\label{prop:spectral-cluster}
If $1\le p \le \ST{d_1}$, then
\[
\norm*{\mathbf{1}_{[K|\mu|,(K+1)|\mu|)}(L^\mu) }_{p\to 2} \le C_p\, |\mu|^{\frac{d_1}2(\frac 1 p - \frac 1 2)}(K+1)^{\frac{d_1}2(\frac 1 p - \frac 1 2)-\frac 1 2}
\]
for all $K\in \N$ and almost all $\mu\in\g_2^*$.
\end{proposition}

The case $p=1$ of \cref{thm:restriction-type} will follow from a Plancherel estimate for the associated integral kernel.

\begin{lemma}\label{lem:Plancherel}
The following statements hold:
\begin{enumerate}
\item If $\mathcal K_{\ell,j}$ denotes the (group) convolution kernel of $F(L)\chi(2^\ell U)\zeta_j(\mathbf U)$, then
\[
\| \mathcal K_{\ell,j} \|_2 \le C_{A,\chi} \, \delta^{(d_2-1)/2} \, 2^{-\ell d_2/2} \, \|F\|_2 \quad \text{for all }\ell\in\Z.
\]
\item There is some $\ell_0\in\N$, which depends only on the matrices $J_\mu$ of \cref{eq:skew-form-ii}, the inner product $\langle\cdot,\cdot\rangle$ on $\g$, and the compact subset $A\subseteq \R^+$ of \cref{thm:restriction-type} such that
\[
F(L)\chi(2^\ell U)=0 \quad \text{for all } \ell<-\ell_0.
\]
\end{enumerate}
\end{lemma}

\begin{proof}
By \cite[Corollary 8]{MaMue14N}, using the notation of \cref{prop:J_mu-spectral}, 
\begin{equation}
\begin{split}
\| \mathcal K_{\ell,j} \|_2^2
= (2 \pi)^{|\mathbf{r}|_1-(d_1+d_2)} \int_{\mathfrak{g}_{2, r}^*} & \int_0^\infty \sum_{\mathbf{k} \in \mathbb{N}^N}\left|F(s+\eigvp{k}{\mu})\,\chi(2^\ell|\mu|) \, \zeta_j(\mu) \right|^2 \\
& \times \prod_{n=1}^N\left[\left(b_n^\mu\right)^{r_n}
\binom{k_n+r_n-1}{k_n}\right] d \sigma_{r_0}(s) \,d \mu, \label{eq:rest-interp-1}
\end{split}
\end{equation}
where $\mu\mapsto\mathbf{b}^\mu=(b_1^\mu,\dots,b_N^\mu)\in [0,\infty)^N$ is homogeneous of degree 1, $\eigvp{k}{\mu}$ is given by
\[
\eigvp{k}{\mu} = \sum_{n=1}^N \left(2k_n+r_n\right)b_n^\mu,
\]
and $\sigma_{r_0}$ is the Dirac delta at 0 if $r_0 =0$, and
\[
d\sigma_{r_0}(s)=\frac{\pi^{r_0 / 2}}{\Gamma(r_0 / 2)} s^{r_0 / 2} \frac{d s}{s}\quad \text{if } r_0>0.
\]
Note that $b_n^\mu \left(k_n+1\right) \le \eigvp{k}{\mu} \lesssim_A 1$
whenever $s+\eigvp{k}{\mu}\in\supp F\subseteq A$, where $A\subseteq \R^+$ is the fixed compact subset of \cref{thm:restriction-type}. Moreover,
\[
\binom{k_n+r_n-1}{k_n} \sim (k_n+1)^{r_n-1}.
\]
Thus, the right-hand side of \cref{eq:rest-interp-1} can be bounded by a constant times
\begin{equation}\label{eq:Plancherel-computation}
\int_0^\infty \int_{\mathfrak{g}_{2, r}^*} \sum_{\mathbf{k} \in \mathbb{N}^N}\left|F(s+\eigvp{k}{\mu})\,\chi(2^\ell|\mu|)\, \zeta_j(\mu) \right|^2 \prod_{n=1}^Nb_n^\mu \,d \mu \,d \sigma_{r_0}(s).
\end{equation}
We use polar coordinates $\mu = \rho \omega$, where $\rho \in [0,\infty)$ and $|\omega|=1$. Since $\mu\mapsto\mathbf{b}^\mu$ is homogeneous of degree 1, we have $\eigvp{k}{\mu} = \rho \eigvp{k}{\omega}$. Thus, using $\supp (\zeta_j|_{S^{d_2-1}})\subseteq V_j$, \cref{eq:Plancherel-computation} is bounded by a constant times
\begin{align*}
2^{-\ell (d_2+N)}\int_0^\infty \int_{V_j}  \int_0^\infty \sum_{\mathbf{k} \in \mathbb{N}^N} & \left|F(s+\rho\eigvp{k}{\omega})\,\chi(2^\ell\rho)\right|^2 \\
& \times \prod_{n=1}^N b_n^\omega \,\frac{d\rho}{\rho} \,d\sigma(\omega) \,d \sigma_{r_0}(s).
\end{align*}
Substituting $\rho=(\eigvp{k}{\omega})^{-1}\lambda$ in the inner integral, we see that the above term equals
\begin{align*}
2^{-\ell (d_2+N)}
\int_0^\infty \int_{V_j} \int_0^\infty \sum_{\mathbf{k} \in \mathbb{N}^N}\left|F(s+\lambda)\,\chi(2^\ell (\eigvp{k}{\omega})^{-1} \lambda)\right|^2\\
\times \prod_{n=1}^Nb_n^{\omega} \,\frac{d \lambda}\lambda \,d\sigma(\omega) \,d \sigma_{r_0}(s).
\end{align*}
If $2^\ell (\eigvp{k}{\omega})^{-1} \lambda\in\supp\chi$, then $(2k_n+r_n) b_n^\omega \le \eigvp{k}{\omega} \sim 2^\ell \lambda$. Thus, $k_n\lesssim 2^\ell \lambda (b_n^\omega)^{-1}$ for the non-vanishing summands in the above sum over $\mathbf{k}$. Hence, the above term is bounded by a constant times
\[
\delta^{d_2-1} 2^{-\ell d_2}
\int_0^\infty \int_0^\infty \left|F(s+\lambda)\right|^2 \lambda^N \frac{d \lambda}\lambda \,d \sigma_{r_0}(s),
\]
which can be dominated by $\delta^{d_2-1} 2^{-\ell d_2}\|F\|^2_2$ since $F$ is compactly supported.
\end{proof}

Using \cref{lem:Plancherel}, we obtain \cref{thm:restriction-type} for $p=1$.

\begin{proposition}\label{prop:p=1}
The restriction type estimate \cref{eq:restriction-type} holds true for $p=1$, that is,
\[
\norm{ F(L)\chi(2^\ell U) \zeta_j(\mathbf U) }_{1\to 2} \le C_{\chi} \,\delta^{(d_2-1)/2}\, 2^{-\ell d_2/2}\, \|F\|_2 \quad \text{for all }\ell\in\Z.
\]
\end{proposition}

\begin{proof}
If $\mathcal K_{\ell,j}$ denotes the convolution kernel of $F(L)\chi(2^\ell U)\zeta_j(\mathbf U)$, then
\[
\| F(L)\chi(2^\ell U) \zeta_j(\mathbf U) f \|_2 = \| f * \mathcal K_{\ell,j} \|_2 \le  \| f \|_1  \| \mathcal K_{\ell,j} \|_2,
\]
where $*$ denotes the group convolution on $G$.
\end{proof}

It remains to show \cref{eq:restriction-type} for $p=\ST{d_1,d_2}$. To that end, let $\vartheta_p\in [0,1]$ be such that $1/p = (1-\vartheta_p) + \vartheta_p/\ST{d_2}$. We show the following:

\begin{proposition}\label{prop:p=ST}
If $1\le p\le \ST{d_1,d_2}$, then,
\begin{equation}\label{eq:restriction-type-i}
\norm*{ F(L)\chi(2^\ell U) \zeta_j(\mathbf U) }_{p\to 2}
\le C_{p,\chi} \, \delta^{\frac{d_2-1} 2 (1-\vartheta_p)}\,2^{-\ell d_2(\frac 1 p - \frac 1 2)} \, \|F\|_{2^{\ell},2},
\end{equation}
for all $\ell\in\Z$, where $F,\chi,\zeta_j$ are as in \cref{thm:restriction-type}.
\end{proposition}

Since the spectral localization corresponding to the operator $\zeta_j(\mathbf U)$ corresponds to a frequency localization in $\mu$ to a cap of size $\delta$ on the unit sphere, we also need the following restriction estimate. This is a version of the Stein–Tomas restriction estimate, where the Fourier transform is further localized to a cap on the sphere.

\begin{lemma}\label{lem:caps}
If $1\le p\le \ST{d_2}$, then
\begin{equation}\label{eq:ST-delta}
\bigg(\int_{\angle(\omega,\omega_0)<\delta} |\hat f(\omega)|^2 \,d\sigma(\omega)\bigg)^{\frac 1 2} \le C_p \, \delta^{\frac{d_2-1} 2 (1-\vartheta_p)}\,\|f\|_{L^p(\R^{d_2})}
\end{equation}
for all $\omega_0\in S^{d_2-1}$, $\delta>0$ and $f\in \mathcal S(\R^{d_2})$.
\end{lemma}

\begin{proof}
The proof follows by interpolation. For $p=1$, we have
\[
\bigg(\int_{\angle(\omega,\omega_0)<\delta} |\hat f(\omega)|^2 \,d\sigma(\omega)\bigg)^{\frac 1 2} \lesssim
\delta^{\frac{d_2-1}{2}}\,\|f\|_{\infty}.
\]
Interpolating this with the endpoint Stein--Tomas estimate
\[
\|\hat f\|_{L^2(S^{d_2-1})} \lesssim \|f\|_{\ST{d_2}},
\]
we obtain \cref{eq:ST-delta} with a constant of $\delta^{(d_2-1)(1-\vartheta_p)/2}$, where $\vartheta_p\in [0,1]$ satisfies
\[
1/p=(1-\vartheta_p)+\vartheta_p/\ST{d_2}. \qedhere
\]
\end{proof}

\begin{remark}
In the endpoint $p=\ST{d_2}$, we cannot expect any gain in $\delta$, which can be seen by a Knapp example: Suppose that $\smash{\hat f}$ is supported in a box of size $\sim\delta$ tangential to the $\delta$-cap and size $\sim \delta^2$ transversal to the $\delta$-cap. Then
\begin{align*}
\|f\|_{\ST{d_2}} & \sim \delta^{d_2+1} \delta^{-(d_2+1)/\ST{d_2}} \\
& = \delta^{(d_2+1)/\ST{d_2}'} = \delta^{\frac{d_2-1}{2}} \sim \|\hat f\|_{L^2(S^{d_2-1})}.
\end{align*}
\end{remark}

\begin{proof}[Proof of \cref{prop:p=ST}]
Suppose that $1\le p\le \ST{d_1,d_2}$. Let $f\in\mathcal S(G)$. Using \cref{eq:joint-calculus-ii} and the Plancherel theorem on $L^2(\g_2)$, we obtain
\begin{align}
\norm*{ F(L)\chi(2^\ell U) \zeta_j(\mathbf U) f  }^2_{L^2(G)}
& \sim \int_{\g_2^*} {\norm*{F(L^\mu)\chi(2^\ell |\mu|) \zeta_j(\mu) f^\mu}_{L^2(\g_1)}^2} \, d\mu \notag \\
& \lesssim \int_{|\mu|\sim 2^{-\ell}}  |\zeta_j(\mu)|^2\, {\norm*{F(L^\mu) f^\mu}_{L^2(\g_1)}^2} \, d\mu. \label{eq:proof-restriction-i}
\end{align}
Moreover, orthogonality on $L^2(\g_1)$ yields
\begin{align}
& \norm*{F(L^\mu) f^\mu}_{L^2(\g_1)}^2
 = \sum_{K=0}^\infty {\norm*{F|_{[K|\mu|,(K+1)|\mu|)}(L^\mu) f^\mu}_{L^2(\g_1)}^2} \notag \\
& \le \sum_{K=0}^\infty {\norm*{F|_{[K|\mu|,(K+1)|\mu|)}}_\infty^2}\,  \norm*{\mathbf{1}_{[K|\mu|,(K+1)|\mu|)}(L^\mu) f^\mu}_{L^2(\g_1)}^2. \label{eq:proof-restriction-ii}
\end{align}
We may assume that $K|\mu|\lesssim_A 1$ and $(K+1)|\mu|\gtrsim_A 1$ since $F|_{[K|\mu|,(K+1)|\mu|)}=0$ for $\left[K|\mu|,(K+1)|\mu|\right) \cap A = \emptyset$. Then, for $|\mu|\sim 2^{-\ell}$, we have $K\lesssim_A |\mu|^{-1}\sim 2^\ell$, and the spectral cluster estimate of \cref{prop:spectral-cluster} yields
\begin{align}
\norm*{\mathbf{1}_{[K|\mu|,(K+1)|\mu|)}(L^\mu) f^\mu}_{L^2(\g_1)}
 & \lesssim \left|\mu\right|^{\frac{d_1}2(\frac 1 p - \frac 1 2)}\left(K+1\right)^{\frac{d_1}2(\frac 1 p - \frac 1 2)-\frac 1 2} \left\| f^\mu \right\|_{L^p(\g_1)}  \notag \\
 & \lesssim \left|\mu\right|^{\frac 1 2} \left\| f^\mu \right\|_{L^p(\g_1)}. \label{eq:proof-restriction-iii}
\end{align}
Moreover, for $|\mu|\sim 2^{-\ell}$, we have
\[
|\mu| \sum_{K=0}^\infty {\norm*{F|_{[K|\mu|,(K+1)|\mu|)}}_\infty^2} \sim \norm*{F}_{2^{\ell},2}^2.
\]
Combining this with \cref{eq:proof-restriction-i}, \cref{eq:proof-restriction-ii} and \cref{eq:proof-restriction-iii}, we get
\begin{equation}\label{eq:proof-restriction-v}
\norm*{ F(L)\chi(2^\ell U) f }^2_{L^2(G)} \lesssim \norm*{F}_{2^{\ell},2}^2 \int_{|\mu|\sim 2^{-\ell}} |\zeta_j(\mu)|^2 \left\|f^\mu\right\|_{L^p(\g_1)}^2 d\mu .
\end{equation}
Since $2/p\ge 1$, Minkowski's integral inequality yields
\begin{align}
& \int_{|\mu|\sim 2^{-\ell}} |\zeta_j(\mu)|^2 \left\|f^\mu\right\|_{L^p(\g_1)}^2 d\mu \notag \\
& \le \bigg(\int_{\g_1}\bigg( \int_{|\mu|\sim 2^{-\ell}} |\zeta_j(\mu)|^2\, |f^\mu(x)|^2 \,d\mu\bigg)^{\frac p 2}\,dx \bigg)^{\frac 2 p}.\label{eq:proof-restriction-vi}
\end{align}
We write $f_{x}=f(x,\cdot)$. Let $\widehat\cdot$ denote the Fourier transform on $\g_2$. Recall that $\zeta_j|_{S^{d_2-1}}$ is supported on a cap $V_j$ of size $\delta$. Using polar coordinates and applying the restriction estimate from \cref{lem:caps}, we get
\begin{align}
\int_{|\mu|\sim 2^{-\ell}} |f^\mu(x)|^2 \,d\mu
 & = \int_{r\sim 2^{-\ell}} \int_{V_j} | \widehat{f_{x}}(r\omega)|^2 \, d\sigma(\omega) \, r^{d_2-1} \, dr\notag\\
 & = \int_{r\sim 2^{-\ell}} \int_{V_j} \big| \big(f_{x}(r^{-1}\,\cdot\,)\big)^\wedge(\omega)\big|^2 \, d\sigma(\omega)\, r^{-d_2-1} \, dr \notag \\
 & \lesssim \delta^{(d_2-1)(1-\vartheta_p)} \int_{r\sim 2^{-\ell}} \left\|f_{x}(r^{-1}\,\cdot\,)\right\|_{L^p(\g_2)}^2\,r^{-d_2-1} \, dr\notag \\
 & = \delta^{(d_2-1)(1-\vartheta_p)} \int_{r\sim 2^{-\ell}} r^{2d_2(\frac 1 p -\frac 1 2)-1}  \, dr \left\|f_{x}\right\|_{L^p(\g_2)}^2\notag \\
 & \sim \delta^{(d_2-1)(1-\vartheta_p)} \, 2^{-2\ell d_2(\frac 1 p -\frac 1 2)}  \left\|f_{x}\right\|_{L^p(\g_2)}^2.\notag 
\end{align}
In combination with \cref{eq:proof-restriction-v} and \eqref{eq:proof-restriction-vi}, we obtain
\[
\norm*{ F(L)\chi(2^\ell U) \zeta_j(\mathbf U) f}_{L^2(G)}
\lesssim_{A,\chi} \delta^{\frac{d_2-1} 2(1-\vartheta_p)} \, 2^{-\ell d_2(\frac 1 p -\frac 1 2)}\, \|F\|_{2^{\ell},2} \, \|f\|_{L^p(G)},
\]
which yields \cref{eq:restriction-type-i}.
\end{proof}

\section{Reduction of \texorpdfstring{\cref{thm:main}}{Theorem 1.1} to dyadic spectral multipliers}\label{sec:reduction}

\subsection{Dyadic spectral multipliers}

We reduce the proof of the statements (\ref{main-(1)}) and (\ref{main-(2)}) of \cref{thm:main} to the case of spectral multipliers whose Fourier transforms are supported on dyadic scales. To that end, let $\phi\in C_c^\infty(\R)$ be an even function supported in $[-2,-1]\cup[1 ,2]$. Define $\chi(\lambda)=\phi(\lambda)-\phi(2\lambda)$ and put $\chi_j(\lambda)=\chi(2^{-j}\lambda)$ for $j\in\N$. Additionally, let $\smash{\chi_{-1}=1-\sum_{j\ge 0}\chi_j}$. Given a suitable multiplier $F:\R\to\C$, we decompose it as
\[
F = F_{-1} + \sum_{\iota\ge 0} F^{(\iota)}, \quad \text{where } F^{(\iota)} := (\hat F \chi_\iota)^\vee,
\]
and where $\widehat\cdot$ and $\cdot^\vee$ denote the Fourier transform and its inverse on $\R$, respectively. Then, to prove \cref{thm:main}, it suffices to show \cref{eq:reduced} of the next theorem, which is Corollary 6.2 of \cite{Ni24}.

\begin{theorem}\label{thm:reduction}
Let $G$ be a two-step stratified Lie group and $L$ be a sub-Laplacian as in \cref{eq:sub-Laplacian}. Let $p_{*}\in[1,2]$ and $s>1/2$. Suppose that for all $1\le p\le p_{*}$ there exists some $\varepsilon>0$ such that
\begin{equation}\label{eq:reduced}
\| F^{(\iota)}(\sqrt L) \|_{p\to p} \le C_{p,s} 2^{-\varepsilon\iota} \|F\|_{L^2_s}
\quad \text{for all } \iota \in\N
\end{equation}
and all even bounded Borel functions $F\in L^2_s(\R)$ supported in $[-2,-\frac 1 2]\cup[\frac 1 2,2]$. Then the statements (\ref{main-(1)}) and (\ref{main-(2)}) of \cref{thm:main} hold for all $1\le p\le p_{*}$.
\end{theorem}

\subsection{Finite propagation speed}

The convolution kernels of the dyadic pieces $\smash{F^{(\iota)}(\sqrt L)}$ in \cref{thm:reduction} are compactly supported with respect to a control distance, which we will now discuss. 

The Lie group $G \cong \mathbb{R}^{d_1} \times \mathbb{R}^{d_2}$ admits a natural control distance associated with the sub-elliptic operator $L$. This distance, which is the \textit{Carnot--Carathéodory distance} $d_{\mathrm{CC}}$, is defined for any two points $g, h \in G$ as the infimum of lengths of curves in $G$ connecting $g$ and $h$ that are horizontal with respect to the vector fields $X_1, \dots, X_{d_1}$ (see \cite[Section~III.4]{VaSaCo92}). The Chow-Rashevskii theorem \cite[Proposition~III.4.1]{VaSaCo92} ensures that $d_{CC}$ is a metric on $G$ inducing the Euclidean topology. Moreover, as shown in \cite[Lemma~1.1]{FoSt82} (see also \cite[Section~5]{Ni24}), the Carnot--Carathéodory distance satisfies
\begin{equation}\label{eq:equivalence-norms}
d_{\mathrm{CC}}(g, h) \sim \|g^{-1}h\|_G \quad \text{for all } g, h \in G,
\end{equation}
where $\| \cdot \|_G$ is the norm defined by $\|(x, u)\|_G := |x| + |u|^{1/2}$ for $(x, u) \in \mathbb{R}^{d_1} \times \mathbb{R}^{d_2}$.

For $(x, u) \in G$, let $B_R^{d_{\mathrm{CC}}}(x, u)$ denote the ball of radius $R > 0$ centered at $(x, u)$ with respect to the distance $d_{\mathrm{CC}}$. Note that $d_{\mathrm{CC}}$ is left-invariant with respect to the group multiplication, that is, $d_{\mathrm{CC}}(ag, ah) = d_{\mathrm{CC}}(g, h)$ for all $a, g, h \in G$.
Moreover, since the norm $\|\cdot\|_G$ is homogeneous with respect to the family of dilations $(\delta_R)_{R > 0}$ given by $\delta_R(x, u) = (Rx, R^2u)$, the volume of a ball $B_R^{d_{\mathrm{CC}}}(x, u)$ satisfies
\begin{equation}\label{eq:cc-volume}
|B_R^{d_{\mathrm{CC}}}(x, u)| = R^Q |B_1^{d_{\mathrm{CC}}}(0)| \sim R^Q,
\end{equation}
where $Q = d_1 + 2d_2$, and $0 \in G$ is the neutral element of the group multiplication under exponential coordinates. In particular, the metric space $(G, d_{\mathrm{CC}})$ equipped with the Lebesgue measure is a space of homogeneous type with homogeneous dimension $Q$. Moreover, by \eqref{eq:equivalence-norms}, there exists a constant $C > 0$ such that
\begin{equation}\label{eq:C}
B_R^{d_{\mathrm{CC}}} = 
B_R^{d_{\mathrm{CC}}}(0) \subseteq B_{CR}(0) \times B_{CR^2}(0) \subseteq \mathbb{R}^{d_1} \times \mathbb{R}^{d_2},
\end{equation}
where $B_{CR}(0)$ and $B_{CR^2}(0)$ are balls with respect to the Euclidean distance.

By \cite{Me84} (or alternatively \cite[Corollary 6.3]{Mue04}), any sub-Laplacian $L$ on a two-step stratified Lie group satisfies the \textit{finite propagation speed property} with respect to the Carnot--Carathéodory distance $d_{\mathrm{CC}}$ on $G$.

\begin{lemma}\label{lem:finite-prop}
If $f,g\in L^2(G)$ are supported in open subsets $U,V\subseteq G$, then
\[
\langle \cos(t\sqrt L)f,g\rangle_{L^2(G)} = 0\quad\text{for all } |t|< d_{\mathrm{CC}}(U,V),
\]
where $d_{\mathrm{CC}}(U,V) = \inf\{\rho(u,v):u\in U,v\in V \}$ is the distance of $U,V\subseteq G$.
\end{lemma}

Let $F : \mathbb{R} \to \mathbb{C}$ be an even bounded Borel function. Since $\chi_\iota$ is also even, the Fourier inversion formula yields  
\[
F^{(\iota)}(\sqrt{L}) = \frac{1}{2\pi} \int_{2^{\iota-1} \leq |\tau| \leq 2^{\iota+1}} \chi_\iota(\tau) \hat{F}(\tau) \cos(\tau \sqrt{L}) \, d\tau.
\]
Given the finite propagation speed property of $L$, if $f \in L^2(G)$ is supported in the ball $B_R^{d_{\mathrm{CC}}}(x, u)$ with $R = 2^{\iota}$, then the support of $F^{(\iota)}(\sqrt{L})f$ is contained within the larger ball $B_{3R}^{d_{\mathrm{CC}}}(x, u)$. Consequently, if $\mathcal{K}^{(\iota)}$ denotes the convolution kernel of the operator $F^{(\iota)}(\sqrt{L})$, that is, $F^{(\iota)}(\sqrt{L}) f = f * \mathcal{K}^{(\iota)}$, then $\mathcal{K}^{(\iota)}$ is supported in the ball $B_{2R}^{d_{\mathrm{CC}}}$ centered at the origin. The volume of this ball scales like $R^{d_1+2d_2}$.

\section{A localization on the first layer}\label{sec:1-layer}

As before, let $L=-(X_1^2+\dots+X_{d_1}^2)$ be a sub-Laplacian on a two-step stratified Lie group~$G$. For the results of this section, we do not require Assumptions \ref{assumptionA} and \ref{assumptionB}. Instead, we only require that 
\[
\mu\mapsto\dim \ker J_\mu \text{ is constant on } \g_2^*\setminus\{0\},
\]
where $J_\mu$ are the matrices from Section \ref{sec:basics}. Recall that, by Lemma \ref{lem:proj-radical}, this implies that there is a constant $\kappa>0$ such that
\begin{equation}\label{eq:proj-radical-ii}
\smash{\|P_0^\mu - P_0^{\mu'}\| \le \kappa\, |\mu-\mu'| \quad\text{for all } \mu,\mu' \in \g_2^*\setminus\{0\}},
\end{equation}
where $P_0^\mu$ is the orthogonal projection onto $\ker J_\mu$.
We may assume that $\kappa>1$.

\subsection{Convolution kernels}

Let $F:\mathbb{R} \rightarrow \mathbb{C}$ be an even bounded Borel function. Let $F^{(\iota)}$ be as in Section~\ref{sec:reduction}. Recall from Section~\ref{sec:restriction} that the operator $F^{(\iota)}(\sqrt L)$ can be decomposed as
\[
F^{(\iota)}(\sqrt L) = \sum_{\ell \geq -\ell_0} \sum_{j \in I} F^{(\iota)}(\sqrt L) \chi(2^{\ell} U) \zeta_j(\mathbf{U}).
\]
Here, $\mathbf{U} = (-iU_1, \dots, -iU_{d_2})$ and $U = |\mathbf{U}|^{1/2}$, where $U_1, \dots, U_{d_2}$ forms a basis for the second layer $\mathfrak{g}_2$. Moreover, $\chi \in C_c^{\infty}(\mathbb{R}^+)$ and $(\chi(2^{\ell} \cdot))_{\ell\in\Z}$ forms a dyadic partition of unity of the real axis. The constant $\ell_0 \in \mathbb{N}$ is given by Lemma~\ref{lem:Plancherel}. The functions $\zeta_j: \mathfrak{g}_2^* \setminus \{0\} \rightarrow \mathbb{R}$ are homogeneous of degree 0 and localize to caps $V_j$ of size $\delta$ on the unit sphere $S^{d_2-1} \subseteq \mathfrak{g}_2^*$, where they form a partition of unity.

The operators $F^{(\iota)}(\sqrt L) \chi(2^{\ell} U) \zeta_j(\mathbf{U})$ possess a convolution kernel which can be  written down in terms of the Fourier transform on the second layer and Laguerre polynomials. We use the notation of \cref{prop:J_mu-spectral}. Recall that \cref{prop:J_mu-spectral} provides the spectral decomposition
\[
-J_\mu^2 = \sum_{n=1}^N (b_n^\mu)^2 P_n^\mu
\]
for all $\mu$ in a Zariski open subset $\g_{2,r}^*\subseteq\g_2^*$, where $P_n^\mu$ is an orthogonal projection of rank $2r_n$. Furthermore, for $\lambda>0$ and $m\in\N\setminus\{0\}$, let $\smash{\varphi_k^{(\lambda,m)}}$ be the $\lambda$-rescaled Laguerre function given by
\begin{equation}\label{M-eq:Laguerre}
\varphi_k^{(\lambda,m)}(z) = \lambda^m L_k^{m-1}\big(\tfrac 1 2 \lambda |z|^2\big) \,e^{-\frac 1 4 \lambda |z|^2},\quad z\in\R^{2m},
\end{equation}
where $L_k^{m-1}$ is the $k$-th Laguerre polynomial of type $m-1$. Since $P_n^\mu$ projects onto a subspace of dimension $2r_n$, with a slight abuse of notation, we also write
\[
\varphi_k^{(\lambda,r_n)}(P_n^\mu x) = \lambda^m L_k^{m-1}\big(\tfrac 1 2 \lambda |P_n^\mu x|^2\big) \,e^{-\frac 1 4 \lambda |P_n^\mu x|^2}.
\]

\begin{lemma}\label{lem:conv-kernel}
If $F\in \mathcal S(\R)$, then the operator $F(L) \chi(2^{\ell} U) \zeta_j(\mathbf{U})$ possesses a convolution kernel $\mathcal K_{\ell,j}\in \mathcal S(G)$, which is given by
\begin{multline}\label{eq:conv-kernel}
\mathcal K_{\ell,j}(x,u) = (2\pi)^{-r_0-d_2} \int_{\g_{2,r}^*} \int_{\ker J_\mu} \sum_{\mathbf{k}\in\N^N} F(|\xi|^2+\eigvp{k}{\mu})\, \chi(2^\ell|\mu|) \, \zeta_j(\mu) \\
 \times \bigg[\prod_{n=1}^N \varphi_{k_n}^{(b_n^\mu,r_n)}(P_n^\mu x)\bigg] e^{i\langle \xi,  P_0^\mu x\rangle} \, e^{i\langle \mu, u\rangle} \, d\xi  \, d\mu,
\end{multline}
where
\[
\eigvp{k}{\mu} = \sum_{n=1}^N \left(2k_n+r_n\right)b_n^\mu ,\quad \mathbf{k}=(k_1,\dots,k_N)\in \N^N.
\]
\end{lemma}

\begin{proof}
This follows from \cite[Proposition 3.4]{Ni25R}.
\end{proof}

\subsection{The role of the decomposition into caps}

The heuristic idea behind the following \cref{prop:1st-layer-Plancherel} is as follows: As outlined in \cref{sec:reduction}, the convolution kernel of the operator $\smash{F^{(\iota)}(\sqrt L)}$ is supported within a ball of size $R \times R^2$, up to constants, where $R = 2^\iota$. Suppose we replace $F$ in \cref{eq:conv-kernel} with the function $\smash{F^{(\iota)}(\sqrt{\cdot\,})}$. Let $\smash{\mathcal{K}^{(\iota)}_{\ell,j}}$ be the convolution kernel of the operator $\smash{F^{(\iota)}(\sqrt{L}) \chi(2^{\ell} U) \zeta_j(\mathbf{U})}$.

Note that $b_n^\mu > 0$ for all $\mu \in \g_2^* \setminus \{0\}$, which follows from our assumption that $\mu \mapsto \dim \ker J_\mu$ is constant on $\g_2^* \setminus \{0\}$. Furthermore, since $\mu \mapsto b_n^\mu$ is homogeneous of degree 1, we have $b_n^\mu \sim |\mu|$. The estimates from \cite[Equations~(1.1.44), (1.3.41), Lemma~1.5.3]{Th93} suggest that the function $\smash{\varphi_{k_n}^{(b_n^\mu, r_n)}}$ in \cref{eq:conv-kernel} is essentially supported in a ball of radius comparable to
\[
(b_n^\mu)^{-1/2}(k_n+1)^{1/2} \sim |\mu|^{-1/2}(k_n+1)^{1/2} \lesssim 2^\ell,
\]
centered at the origin, up to an exponentially decaying term.

Let $R_\ell = 2^\ell$. We distinguish between the cases $-\ell_0 \le \ell \le \iota$, where $R_\ell \le R$, and the case $\ell > \iota$. We consider the case where $-\ell_0 \le \ell \le \iota$. The observation above suggests that we can think of the partial Fourier transform $\smash{(\mathcal{K}^{(\iota)}_{\ell,j})^\mu}$ (as defined in \cref{eq:partial-FT}) as being supported in the set of all $x \in \g_1$ such that
\begin{equation}\label{eq:heuristic-1}
|P_0^\mu x| \lesssim R \quad \text{and} \quad |\bar P^\mu x| \lesssim R_\ell.
\end{equation}
Here, $P_0^\mu x$ is the orthogonal projection of $x$ onto $\ker J_\mu$, and $\bar P^\mu x = x - P_0^\mu x$ is the projection onto the orthogonal complement $(\ker J_\mu)^\perp$. However, this information about the support depends on $\mu$. For instance, if $G$ is the group $N_{3,2}$, then $\ker J_\mu$ is the one-dimensional subspace spanned by the vector $\mu$, see \cref{ex:N32}.

This is where the additional decomposition provided by the functions $\zeta_j$ comes into play. We fix $\mu^j \in S^{d_2-1}$, where $\mu^j \in \zeta_j$ (thought of as the center of the cap). If $\mu \in \supp \zeta_j$, we have
\[
|\bar \mu - \mu^j| < \delta, \quad \text{where} \quad \bar \mu = |\mu|^{-1} \mu.
\]
Thus, using \cref{eq:proj-radical-ii}, if we choose $\delta$ sufficiently small, then \cref{eq:heuristic-1} implies
\[
|P_0^{\mu^j} x| \lesssim R \quad \text{and} \quad |\bar P^{\mu^j} x| \lesssim R_\ell\quad\text{for all }(x, u) \in \supp \mathcal{K}^{(\iota)}_{\ell,j}.
\]
Hence, $\smash{\mathcal{K}^{(\iota)}_{\ell,j}}$ is essentially supported in a ball with volume comparable to $\smash{R^{r_0} R_\ell^{\bar{d}_1} R^{d_2}}$. This idea will be made precise in \cref{prop:rapid-decay-1st-layer}.

Let $C > 0$ be the constant from \cref{eq:C}, that is,
\[
B_R^{d_{\mathrm{CC}}} = B_R^{d_{\mathrm{CC}}}(0) \subseteq B_{CR} \times B_{CR^2} \subseteq \mathbb{R}^{d_1} \times \mathbb{R}^{d_2}.
\]
For any $j\in I$, we fix $\mu^j\in S^{d_2-1}$ with $\mu^j\in\supp\chi_j$. As before, given $x \in \g_1$, let $\smash{P_0^{\mu^j} \! x}$ be the orthogonal projection of $x$ onto $\ker J_{\mu^j}$ and $\smash{\bar P^{\mu^j} \! x=x-P_0^{\mu^j}\! x}$ be the projection onto its orthogonal complement.

\begin{proposition}[Localization on the first layer]\label{prop:rapid-decay-1st-layer}
Let $1 \leq p < 2$. Consider radii $R = 2^\iota$ and $R_\ell = 2^\ell$ for $\iota \in \mathbb{N}$ and $\ell \in \{-\ell_0, \dots, \iota\}$. Let $A \subseteq \mathbb{R}^+$ be a compact subset. Suppose $F \in C_c^\infty(\mathbb{R}^+)$ is supported in $A$, and let $\chi$ and $\zeta_j$ be as before, where we choose $\delta=R/R_\ell$ for the decomposition into caps $V_j$ of size $\delta$. Then, for any $\gamma > 0$, if $x^0 \in \mathfrak{g}_1$ with $|x^0| < CR$ and $f \in L^p(G) \cap L^2(G)$ is supported in
\[
B^{(\ell)}(x^0, \mu^j) = \big\{ (x, u) \in B_{CR} \times B_{CR^2} : |\bar{P}^{\mu^j} (x - x^0)| < CR_\ell \big\},
\]
the $L^p$-norm of the function
\[
g_{\iota, \ell, j} = \mathbf{1}_{B_{3R}^{d_{\mathrm{CC}}}} F(L) \chi(2^\ell U) \zeta_j(\mathbf{U}) f
\]
is rapidly decaying in $R$ outside the ball
\[
\tilde{B}^{(\ell)}(x^0, \mu^j) = \big\{ (x, u) \in B_{3CR} \times B_{9CR^2} : |\bar{P}^{\mu^j} (x - x^0)| < 6\kappa C R_\ell R^\gamma \big\}.
\]
That is, for any $N\in\N$, there is some $C_{N,\gamma}>0$ independent of $x^0$ such that
\begin{equation}\label{eq:rapid-decay-1}
\big\| \, g_{\iota,\ell,j}\big|_{\g\setminus \tilde B^{(\ell)}(x^0,\mu^{j})} \big\|_p \le C_{N,\gamma} R^{-\gamma N} \|F\|_\infty \|f\|_p.
\end{equation}
\end{proposition}

Note that \cref{eq:rapid-decay-1} is formulated for a general multiplier $F$, rather than specifically for $F^{(\iota)}(\sqrt{\,\cdot\,})$. However, the function $g_{\iota, \ell, j}$ is inherently localized by definition to the ball $\smash{B_{3R}^{d_{\mathrm{CC}}}}$ via the cutoff function $\mathbf{1}_{B_{3R}^{d_{\mathrm{CC}}}}$.

The proof of \cref{prop:rapid-decay-1st-layer} relies on the following Plancherel type estimate:

\begin{lemma}\label{prop:1st-layer-Plancherel}
Let $\mathcal{K}_{\ell,j}$ be the convolution kernel of $F(L) \chi(2^\ell U) \zeta_j(\mathbf{U})$. Then
\begin{equation}\label{eq:first-layer-1}
\int_G \mathbf{1}_{E_\ell(\mu^j)}(x) \left|\mathcal{K}_{\ell,j}(x, u)\right|^2  d(x, u) \leq C_{N} \big( R^{-\gamma N} \|F\|_\infty \big)^2,
\end{equation}
where $E_\ell(\mu^j)$ is the set of all $x \in \mathfrak{g}_1$ satisfying
\[
|x| < 4CR \quad \text{and} \quad |\bar{P}^{\mu^j} x| \geq 5\kappa C R_\ell R^\gamma.
\]
\end{lemma}

\begin{proof}[Proof of \cref{prop:1st-layer-Plancherel}]
By \cref{lem:conv-kernel} and Plancherel's theorem, the left-hand side of \eqref{eq:first-layer-1} equals a constant times 
\begin{multline}\label{eq:first-layer-4}
\int_{\g_1} \mathbf{1}_{E_{\ell}(\mu^j)}(x)  \int_{\g_{2,r}^*}  \bigg|
 \int_{\ker J_\mu} \sum_{\mathbf{k}\in\N^N} F(|\xi|^2+\eigvp{k}{\mu})\,\chi(2^\ell|\mu|)\,\zeta_j(\mu) \\
 \times \bigg[\prod_{n=1}^N \varphi_{k_n}^{(b_n^\mu,r_n)}(P_n^\mu x)\bigg] e^{i\langle \xi, P_0^\mu x\rangle}\,d\xi\bigg|^2 \,  d\mu\, dx.
\end{multline}
For $\mu\in\g_2^*$, let $E_{\ell}(\mu)$ be the set of all $x \in\g_1$ such that
\[
|x| < 3CR \quad \text{and} \quad |\bar P^{\mu} x| \ge \kappa C R_\ell R^\gamma .
\]
Suppose that $x\in E_{\ell}(\mu^j)$ and $\mu\in\supp \zeta_j$. Then $|\bar \mu - \mu^j|<\delta$, where use again the notation $\bar\mu=|\mu|^{-1}\mu$. By \cref{eq:proj-radical-ii}, using $\bar P^{\mu}={\mathrm{id}_{\g_1}}-P_0^{\mu}$,
\[
|\bar P^{\mu} x - \bar P^{\mu^j} \! x| =
|P_0^{\mu} x - P_0^{\mu^j} \! x|
\le \kappa \, |\bar \mu - \mu^j|\, |x| \le \kappa \delta  |x| .
\]
Since $\delta=R_\ell/R$ and $|x|\le 3CR$, we obtain
\begin{align*}
|\bar P^{\mu} x|
& \ge |\bar P^{\mu^j} x| - \kappa \delta  |x| \\
& \ge 5\kappa C R_\ell R^\gamma - 4\kappa C R_\ell
  \ge \kappa C R_\ell R^\gamma.
\end{align*}
Thus, $E_{\ell}(\mu^j)\subseteq E_{\ell}(\mu)$ if $\mu\in\supp \zeta_j$. Hence, if we additionally interchange the order of integration, we may bound \cref{eq:first-layer-4} by
\begin{multline}\label{eq:first-layer-5}
\int_{\g_{2,r}^*}  \int_{\g_1} \mathbf{1}_{E_{\ell}(\mu)}(x) \,  \bigg|
 \int_{\ker J_\mu} \sum_{\mathbf{k}\in\N^N} F(|\xi|^2+\eigvp{k}{\mu})\,\chi(2^\ell|\mu|)\,\zeta_j(\mu)  \\
 \times \bigg[\prod_{n=1}^N \varphi_{k_n}^{(b_n^\mu,r_n)}(P_n^\mu x)\bigg] e^{i\langle \xi, P_0^\mu x\rangle}\,d\xi\bigg|^2 \, dx \, d\mu.
\end{multline}
Since $|\bar P^{\mu} x| \ge \kappa C R_\ell R^\gamma$ for $x\in E_{\ell}(\mu)$, \cref{eq:first-layer-5} is bounded by a constant times
\begin{multline}\label{eq:first-layer-6}
(R_\ell R^\gamma)^{-2N}\int_{\g_{2,r}^*}  \int_{\g_1} |\bar P^{\mu} x|^{2N} \,  \bigg|
 \int_{\ker J_\mu} \sum_{\mathbf{k}\in\N^N} F(|\xi|^2+\eigvp{k}{\mu})\,\chi(2^\ell|\mu|)\,\zeta_j(\mu)  \\
 \times \bigg[\prod_{n=1}^N \varphi_{k_n}^{(b_n^\mu,r_n)}(P_n^\mu x)\bigg] e^{i\langle \xi, P_0^\mu x\rangle}\,d\xi\bigg|^2 \, dx \, d\mu.
\end{multline}
Recall that $P_n^\mu$ is an orthogonal projection of rank $2r_n$. Thus, for any $\mu\in \g_{2,r}^*$, there is an orthogonal transformation $R_\mu\in O(d_1)$ such that
\[
P_n^\mu R_\mu=R_\mu P_n\quad \text{for all } n\in\{0,\dots,N\},
\]
where $P_n$ denotes the orthogonal projection from $\R^{d_1}=\R^{r_0}\oplus\R^{2r_1}\oplus \dots\oplus \R^{2r_N}$ onto the $n$-th layer. If we substitute $(v,y)=R_\mu^{-1} x$ with $v\in \R^{r_0}$ and $y\in \R^{d_1-r_0}$ and use Plancherel, we see that \cref{eq:first-layer-6} equals a constant times
\begin{multline}\label{eq:first-layer-7}
(R_\ell R^\gamma)^{-2N} \int_{\g_{2,r}^*} \int_{\R^{r_0}} \int_{\R^{\bar d_1}} |y|^{2N} \, \bigg|\sum_{\mathbf{k}\in\N^N} F(|\xi|^2+\eigvp{k}{\mu})\,\chi(2^\ell|\mu|)\,\zeta_j(\mu) \\ \times \prod_{n=1}^N \varphi_{k_n}^{(b_n^\mu,r_n)}(P_n y)\bigg|^2 \, dy\, d\xi\, d\mu,
\end{multline}
where $\bar d_1=d_1-r_0$. We expand
\[
|y|^{2N}=(|P_1 y|^2+\dots+|P_N y|^2)^{N}.
\]
It is sufficient to show that for all $m_1,\dots,m_N\in\N$ with $m_1+\dots+m_N=N$, every term of the form
\begin{multline}\label{eq:first-layer-8}
(R_\ell R^\gamma)^{-2N} \int_{\g_{2,r}^*} \int_{\R^{r_0}} \int_{\R^{\bar d_1}} \bigg(\prod_{n=1}^N |P_n y|^{2m_n}\bigg) \bigg|\sum_{\mathbf{k}\in\N^N} F(|\xi|^2+\eigvp{k}{\mu})\,\chi(2^\ell|\mu|)\,\zeta_j(\mu) \\ \times \prod_{n=1}^N \varphi_{k_n}^{(b_n^\mu,r_n)}(P_n y)\bigg|^2 \, dy\, d\xi\, d\mu
\end{multline}
is bounded by the right-hand side of \cref{eq:first-layer-1}.
For $m\in\N\setminus\{0\}$ and $\lambda>0$, let
\[
H^{(\lambda)}_{\R^{2m}}=- \Delta_z + \tfrac{\lambda^2} 4 |z|^2
\]
denote the (rescaled) Hermite operator on $\R^{2m}$. By \cite[Proposition 3.3]{ChOu16} (or alternatively \cite{He93}), the Hermite operator satisfies the sub-elliptic estimate
\begin{equation}\label{eq:sub-elliptic}
\norm*{|\cdot|^\beta f}_{L^2(\R^{2m})} \lesssim \lambda^{-\beta} \big\| \big(H_{\R^{2m}}^{(\lambda)}\big)^{\beta/2} f\big\|_{L^2(\R^{2m})}, \quad\beta\ge 0.
\end{equation}
By Equations (1.3.25) and (1.3.42) of \cite{Th93}, the Laguerre functions $\smash{\varphi_k^{(\lambda,m)}}$ are eigenfunctions of the $\lambda$-rescaled Hermite operator, that is,
\[
H^{(\lambda)}_{\R^{2m}} \varphi_k^{(\lambda,m)} = \left(2k+m\right) \lambda \, \varphi_k^{(\lambda,m)},\quad k\in \N.
\]
Using the sub-elliptic estimate \eqref{eq:sub-elliptic} on every block of $\R^{\bar d_1}=\R^{2r_1}\oplus\dots\oplus\R^{2r_N}$, \cref{eq:first-layer-8} can be dominated by a constant times
\begin{multline}\label{eq:first-layer-9}
(R_\ell R^\gamma)^{-2N} \int_{\g_{2,r}^*} \int_{\R^{r_0}} \int_{\R^{\bar d_1}}  \bigg|\sum_{\mathbf{k}\in\N^N} F(|\xi|^2+\eigvp{k}{\mu})\,\chi(2^\ell|\mu|)\,\zeta_j(\mu) \\
\times\prod_{n=1}^N \left(b_n^\mu\right)^{-m_n}\left(\left(2k_n+r_n\right)b_n^\mu\right)^{m_n/2} \varphi_{k_n}^{(b_n^\mu,r_n)}(P_n y)\bigg|^2 \, dy\, d\xi\, d\mu.
\end{multline}
Recall that $b_n^\mu>0$ for all $\mu\in \g_{2,r}^*$. This extends to all $\mu\in \g_2^*\setminus\{0\}$ since we have assumed that $\mu\mapsto\dim \ker J_\mu$ is constant on $\g_2^*\setminus\{0\}$. Thus, the image of the unit sphere $S^{d_2-1}$ under the continuous map $\mu\mapsto \mathbf{b}^\mu=(b_1^\mu,\dots,b_N^\mu)$ is a compact subset of $\smash{(\R^+)^N}$, and
\[
b_n^\mu \sim 1 \quad \text{for all } \mu\in S^{d_2-1} \text{ and } n\in\{1,\dots,N\}.
\]
Recall that the functions $\mu\mapsto b_n^\mu$ are homogeneous of degree 1. Thus, if $2^\ell|\mu|$ lies in the support of $\chi$, then
\[
b_n^\mu = \left|\mu\right| b_n^{\bar \mu} \sim |\mu| \sim 2^{-\ell} = R_\ell^{-1} ,\quad \text{where } \bar \mu = |\mu|^{-1}\mu.
\]
Moreover, if $|\xi|^2+\eigvp{k}{\mu}$ lies in the support of $F$, then 
\[
\left(2k_n + r_n\right) b_n^\mu \le \eigvp{k}{\mu} \lesssim_A 1.
\]
Hence, using orthogonality on $\smash{\R^{\bar d_1}}$ and $m_1+\dots+m_N=N$, \eqref{eq:first-layer-9} can be estimated by a constant times
\begin{multline}\label{eq:first-layer-10}
R^{-2\gamma N} \sum_{\mathbf{k}\in\N^N} \int_{\g_{2,r}^*} \int_{\R^{r_0}} \left|F(|\xi|^2+\eigvp{k}{\mu})\,\chi(2^\ell|\mu|) \, \zeta_j(\mu) \right|^2 \\
\times \prod_{n=1}^N \big\| \varphi_{k_n}^{b_n^\mu}\big\|_{L^2(\R^{2r_n})}^2 \, d\xi \, d\mu.
\end{multline}
By Theorem 1.3.2 and Equation (1.3.42) of \cite{Th93}, we have
\begin{align*}
\big\|\varphi_{k_n}^{b_n^\mu}\big\|_{L^2(\R^{2r_n})}^2
& \sim \left(b_n^\mu\right)^{r_n} \binom{k_n+r_n-1}{k_n} \\
& \sim \left(b_n^\mu\right)^{r_n} \left(k_n+1\right)^{r_n-1}
\lesssim b_n^\mu.
\end{align*}
Thus, if we discard the additional factor $\zeta_j$ in the integral, then \eqref{eq:first-layer-10} can be bounded by a constant times
\begin{equation}\label{eq:plancherel-3}
R^{-2\gamma N} \sum_{\mathbf{k}\in\N^N} \int_{\g_{2,r}^*} \int_{\R^{r_0}} \left|F(|\xi|^2+\eigvp{k}{\mu})\,\chi(2^\ell|\mu|)\right|^2 \prod_{n=1}^N b_n^\mu\, d\xi \, d\mu.
\end{equation}
Using polar coordinates in $\rho=|\xi|$ and substituting $s=\rho^2$, \cref{eq:plancherel-3} equals a constant times
\begin{equation}\label{eq:plancherel-4}
R^{-2\gamma N} \sum_{\mathbf{k}\in\N^N} \int_{\g_{2,r}^*} \int_0^\infty \left|F(s+\eigvp{k}{\mu})\,\chi(2^\ell|\mu|)\right|^2 s^{r_0/2-1} \prod_{n=1}^N b_n^\mu\, ds \, d\mu.
\end{equation}
Up to the additional factor $R^{-2\gamma N}$ and a constant, this is \cref{eq:Plancherel-computation}. Hence, by the same arguments as in the proof of \cref{prop:p=1}, we may bound \cref{eq:plancherel-4} by a constant times $R^{-2\gamma N} \|F\|^2_2$, which finishes the proof of \cref{prop:1st-layer-Plancherel}.
\end{proof}

\begin{proof}[Proof of \cref{prop:rapid-decay-1st-layer}]
We interpolate between an $L^1$ and an $L^2$-estimate. For the $L^2$-estimate, we discard all cut-off functions, which yields
\begin{equation}\label{eq:error-L2}
\big\| \, g_{\iota,\ell,j}\big|_{\g\setminus \tilde B^{(\ell)}(x^0,\mu^{j})} \big\|_2
 \lesssim \|F\|_\infty \|f\|_2.
\end{equation}
Note that this bound is uniform in $j\in I$. For the $L^1$-estimate, we consider the sets
\[
A_\ell(x',\mu^j) := \big\{ (x,u)\in B_{3R}^{d_{\mathrm{CC}}} : |\bar P^{\mu^j} \! (x-x')| \ge 5\kappa C R_\ell R^\gamma \big\},\quad x'\in\g_1.
\]
Note that a given point $(x,u)\in G$ lies in $A_\ell(x',\mu^j)$ whenever
\[
(x,u)\in (\g\setminus \tilde B^{(\ell)}(x^0,\mu^{j}))\cap B_{3R}^{d_{\mathrm{CC}}} \quad \text{and}\quad (x',u')\in \supp f.
\]
Let $\mathcal K_{\ell,j}$ be the convolution kernel of $F(L) \chi(2^\ell U) \zeta_j(\mathbf U)$, that is,
\begin{align*}
F(L) \chi(2^\ell U) \zeta_j(\mathbf U) f(x,u)
& =  f * \mathcal K_{\ell,j} (x,u) \\
& = \int_G f(x',u') \, \mathcal K_{\ell,j}\big((x',u')^{-1}(x,u)\big)\,d(x',u').
\end{align*}
Using Fubini's theorem, we get
\begin{align*}
& \big\| \, g_{\iota,\ell,j}\big|_{\g\setminus \tilde B^{(\ell)}(x^0,\mu^{j})} \big\|_1 \\
& \le \int_G \mathbf{1}_{(\g\setminus \tilde B^{(\ell)}(x^0,\mu^{j}))\cap B_{3R}^{d_{\mathrm{CC}}}}(x,u) \,\big| f * \mathcal  K_{\ell,j}(x,u)\big| \,d(x,u)  \\
& \le \int_{\supp f} \int_{A_\ell(x',\mu^j)} \big|f(x',u')\big| \, \big|\mathcal K_{\ell,j}\big((x',u')^{-1}(x,u)\big)\big| \,d(x,u)\,d(x',u')  \\
& \le \|f\|_1 \sup_{(x',u')\in \supp f} \int_{A_\ell(x',\mu^j)} \big|\mathcal K_{\ell,j}\big((x',u')^{-1}(x,u)\big)\big| \,d(x,u). 
\end{align*}
Recall that $|B_{3R}^{d_{\mathrm{CC}}}|\sim R^Q$ by \cref{eq:cc-volume}, where $Q=d_1+2d_2$ denotes the homogeneous dimension of $G$. Hence, using the Cauchy--Schwarz inequality, we obtain
\begin{align}
& \int_{A_\ell(x',\mu^j)} \big|\mathcal K_{\ell,j}\big((x',u')^{-1}(x,u)\big)\big| \,d(x,u) \notag \\
& \lesssim R^{Q/2} \bigg( \int_{A_\ell(x',\mu^j)} \big|\mathcal K_{\ell,j}\big(x-x',u-u'-\tfrac 1 2 [x',x]\big) \big|^2 \,d(x,u)  \bigg)^{1/2}.\label{eq:rapid-decay-3}
\end{align}
Let $E_\ell(\mu^j)$ be defined as in \cref{prop:1st-layer-Plancherel}, that is, $E_\ell(\mu^j)$ is the set of all $x \in\g_1$ such that
\[
|x| < 4CR \quad \text{and} \quad |\bar P^{\mu^j} \! x| \ge 5\kappa C R_\ell R^\gamma .
\]
Since $|x'|<CR$ by assumption, we have $|x-x'|<4CR$ for $(x,u)\in B_{3R}^{d_{\mathrm{CC}}}$. Thus, $x\in A_\ell(x',\mu^j)$ implies $x-x'\in E_\ell(\mu^j)$, so by using \cref{prop:1st-layer-Plancherel}, the right-hand side of \cref{eq:rapid-decay-3} can be bounded by
\[
R^{Q/2} \bigg( \int_G \mathbf{1}_{E_\ell(\mu^j)}(x)\, \big|\mathcal K_{\ell,j}(x,u) \big|^2 \,d(x,u)  \bigg)^{1/2}
\lesssim_N R^{Q/2-\gamma N} \|F\|_2.
\]
Altogether, we end up with
\begin{equation}\label{eq:error-L1}
\big\| \, g_{\iota,\ell,j}\big|_{\g\setminus \tilde B^{(\ell)}(x^0,\mu^{j})} \big\|_1
\lesssim_N R^{Q/2-\gamma N} \|F\|_2 \|f\|_1.
\end{equation}
Finally, interpolation between \cref{eq:error-L1} and \cref{eq:error-L2} yields \cref{eq:rapid-decay-1}.
\end{proof}

\section{A localization on the second layer}\label{sec:2-layer}

Up to know, we did not require Assumptions \ref{assumptionA} and \ref{assumptionB} in the previous sections. From now on, we will require both assumptions. Let $\psi\in C_c^\infty(\R^+)$. By \cref{prop:rapid-decay-1st-layer}, for $\ell\le \iota$, the convolution kernel $\smash{\mathcal K^{(\iota)}_{\ell,j}}$ of the operator $\smash{F^{(\iota)}(\sqrt L)\chi(2^\ell U)\zeta_j(\mathbf U)}$ is essentially supported in the set of all $(x,u)\in G$ such that
\[
|P_0^{\mu^j} x| \lesssim R , \quad |\bar P^{\mu^j} x| \lesssim R_\ell\quad \text{and} \quad |u| \lesssim R^2,
\]
where $R=2^\iota$ and $R_\ell=2^\ell$. 
In this section, we show that we actually even have the stronger localization 
\[
|u| \lesssim R_\ell R
\]
on the second layer, which means that $\mathcal{K}^{(\iota)}_{\ell,j}$ is essentially supported in a ball with volume comparable to $\smash{R^{r_0} R_\ell^{\bar{d}_1} R_\ell^{d_2}R^{d_2}}$. 

\begin{proposition}\label{prop:rapid-decay-2nd-layer}
Suppose that $1 \leq p < 2$. Let $R = 2^\iota$ and $R_\ell = 2^\ell$ for $\iota \in \mathbb{N}$ and $\ell \in \{-\ell_0, \dots, \iota\}$. Let $F:\R\to \C$ be a bounded Borel function, and let $\chi$, $\zeta_j$, and $\mu^j\in \supp \zeta_j|_{S_{d_2-1}}$ be as in \cref{prop:rapid-decay-1st-layer}, where $\delta=R/R_\ell$. Moreover, let  $\psi\in C_c^\infty(\R^+)$. Given $\gamma>0$ and  $C_1,C_2>0$, we consider the balls
\begin{align*}
A^{(\ell)}(\mu^{j}) &= \big\{ (x,u) \in B_{C_1 R}\times B_{C_2 R^2} :  |\bar P^{\mu^j} \! x | < C_1 R_\ell \big\}\quad\text{and}\\
\tilde A^{(\ell)}(\mu^{j}) &= \big\{ (x,u) \in B_{3C_1 R}\times B_{9C_2 R^{2+\gamma}} : |\bar P^{\mu^j} \! x | < 6\kappa C_1 R_\ell R^\gamma \big\}.
\end{align*}
Let $u^0\in\g_2$ with $|u^0|<C_2R^2$.
Then there is a constant $\bar C>0$ depending only on $C_1,C_2,\kappa$ such that, if $f\in L^p(G)\cap L^2(G)$ is a function supported in the set
\[
A^{(\ell)}(\mu^{j},u^0) = \big\{ (x,u) \in A^{(\ell)}(\mu^{j}) : |u- u^0| < C_2 R_\ell R \big\},
\]
then the $L^p$-norm of the function
\[
h_{\iota,\ell,j} = \mathbf{1}_{\tilde A^{(\ell)}(\mu^j)} (F^{(\iota)}\psi)(\sqrt L) \chi(2^\ell U)\zeta_j(\mathbf U) f
\]
is rapidly decaying in $R$ outside the set
\[
\tilde A^{(\ell)}(\mu^{j},u^0)  = \big\{ (x,u) \in \tilde A^{(\ell)}(\mu^{j}) : |u- u^0| < 2\bar C R_\ell R^{1+\gamma} \big\}.
\]
That is, for any $N\in\N$, there is some $C_{N,\gamma,\psi,\chi}>0$ independent of $u^0$ such that
\begin{equation}\label{eq:rapid-decay-ii}
\big\| \,h_{\iota,\ell,j}\big|_{\g\setminus\tilde A^{(\ell)}(\mu^{j},u^0)} \big\|_p \le C_{N,\gamma,\psi,\chi} R^{-\gamma N} \|F^{(\iota)}\|_\infty \|f\|_p.
\end{equation}
\end{proposition}

The proof of \cref{prop:rapid-decay-2nd-layer} is based on the following second layer weighted Plancherel estimate:

\begin{lemma}\label{lem:Plancherel-2nd-layer}
Let $F:\R\to\C$ be a bounded Borel function supported in a compact subset $A\subseteq \R^+$ and $\chi\in C_c^\infty(\R^+)$. If $\mathcal K_{\ell,j}$ denotes the convolution kernel associated with $F(L) \chi(2^\ell U) \zeta_j(\mathbf U)$, then
\begin{equation}\label{eq:Plancherel-2nd-layer}
\int_{G} | |u|^\alpha \, \mathcal  K_{\ell,j}(x,u) |^2 \,dx \,du \leq C_{A,\chi} \, R_\ell^{d_2-2\alpha} \|F\|_{L_2^\alpha}^2
\end{equation}
for all $\alpha\ge 0$, $\ell \ge -\ell_0$, and $j\in I$.
\end{lemma}

\begin{proof}
Let $\mathcal K_{\ell}$ be the convolution kernel associated with $F(L) \chi(2^\ell U)$.
By \cite[Lemma~10]{Ma15}, we have
\[
\int_{G} | |u|^\alpha \, \mathcal  K_{\ell}(x,u) |^2 \,d(x,u) \lesssim R_\ell^{d_2-2\alpha} \|F\|_{L_2^\alpha}^2.
\]
Let $\mathcal K_{\ell}^\mu$ denote again the partial Fourier transform of $\mathcal  K_{\ell}$ with respect to the second layer. Suppose $\alpha\in\N$. Let $j_0\in J$. Since the sets $\supp\chi_j$ have only bounded overlap, we get
\begin{align*}
\int_{G} | |u|^\alpha \, \mathcal  K_{\ell,j_0}(x,u) |^2 \,d(x,u)
& \sim \int_{\g_{2,r}^*} \int_{\g_1} | \partial_\mu^\alpha K_{\ell,j_0}^\mu(x)|^2\, dx \,d\mu \\
& \le \sum_{j\in I} \int_{\g_{2,r}^*} \int_{\g_1} | \partial_\mu^\alpha K_{\ell,j}^\mu(x)|^2\, dx \,d\mu\\
& \sim \int_{\g_{2,r}^*} \int_{\g_1} | \partial_\mu^\alpha K_{\ell}^\mu(x)|^2\, dx \,d\mu \\
& \sim \int_{G} | |u|^\alpha \, \mathcal  K_{\ell}(x,u) |^2 \,d(x,u) 
 \lesssim R_\ell^{d_2-2\alpha} \|F\|_{L_2^\alpha}^2.
\end{align*}
This finishes the proof of \cref{eq:Plancherel-2nd-layer}.
\end{proof}

\begin{remark}
By estimating the left side of \cref{eq:Plancherel-2nd-layer} by the sum over all caps on the sphere $S^{d_2-1}\subseteq \g_2^*$, we lose a factor of $\delta^{d_2-1}$. However, this loss is negligible, as it is compensated by the rapid decay $R^{-\gamma N}$ of \cref{prop:rapid-decay-2nd-layer}. To recover the additional factor of $\delta^{d_2-1}$ in \cref{eq:Plancherel-2nd-layer}, one can directly adapt the proof of Lemma~10 in \cite{Ma15} to take the additional localization to the cap $U_j$ into account.

Alternatively, interpolating between the Plancherel estimate of \cref{lem:Plancherel} and \cref{eq:Plancherel-2nd-layer} by choosing $\alpha$ sufficiently large, one can derive \cref{eq:Plancherel-2nd-layer} with an additional factor $\delta^\beta$ for any $\beta<d_2-1$.
\end{remark}

\begin{remark}
Actually, \cite{Ma15} additionally requires $\supp \chi \subseteq [\frac 1 2,2]$. We can pass to arbitrary $\chi\in C_c^\infty(\R^+)$ by using
\[
F(L) \chi(2^\ell U) \zeta_j(\mathbf U) f_r=\big(F(r^2 L) \chi(r^2\, 2^\ell U) \zeta_j(\mathbf U) f\big)_r,
\]
where $f_r(x,u)=f(rx,r^2u)$. However, this is not crucial for the later proof of \cref{thm:reduction}, since we choose $\supp \chi \subseteq [\frac 1 2,2]$ in the dyadic decomposition for the operator $U$.
\end{remark}

\begin{proof}[Proof of \cref{prop:rapid-decay-2nd-layer}]
Similar to the proof of \cref{prop:rapid-decay-1st-layer}, we again interpolate between an $L^1$ and an $L^2$-estimate. For the $L^2$-estimate, we neglect all cut-off functions to obtain
\begin{equation}\label{eq:interpolation-L2}
\big\| \,h_{\iota,\ell,j}\big|_{\g\setminus\tilde A^{(\ell)}(\mu^{j},u^0)} \big\|_2
\lesssim_{\psi,\chi} \|F^{(\iota)}\|_\infty \|f\|_2.
\end{equation}
The $L^1$-estimate follows from the weighted Plancherel estimate of \cref{lem:Plancherel-2nd-layer}. We consider the sets
\[
\bar A_\ell(\mu^j,u') := \big\{ (x,u) \in \tilde A^{(\ell)}(\mu^j) : |u-u'| \ge \bar C R_\ell R^{1+\gamma} \big\},\quad u'\in \g_2.
\]
If the points $(x,u),(x',u')\in G$ satisfy
\[
(x,u)\in (\g\setminus\tilde A^{(\ell)}(\mu^{j},u^0))\cap\tilde A^{(\ell)}(\mu^j)
\quad\text{and}\quad
(x',u')\in \supp f,
\]
then
\[
|u-u_0| \ge 2\bar C R_\ell R^{1+\gamma},\quad
|u'-u_0| < C_2 R_\ell R
\]
and, if we choose $\bar C\ge \tilde C$, in particular $(x,u)\in \bar A_\ell(\mu^j,u')$.

Let $\mathcal K_{\ell,j}^{(\iota)}$ be the convolution kernel associated with $\smash{(F^{(\iota)}\psi)(\sqrt L)\chi(2^\ell U)\zeta_j(\mathbf U)}$. Then
\begin{align*}
 & (F^{(\iota)}\psi)(\sqrt L)\chi(2^\ell U)\zeta_j(\mathbf U) f(x,u)
 = f * \mathcal K_{\ell,j}^{(\iota)}(x,u) \\
& = \int_G f(x',u') \,\mathcal K_{\ell,j}^{(\iota)}\big((x',u')^{-1}(x,u)\big)\,d(x',u'),\quad (x,u)\in G.
\end{align*}
By Fubini's theorem,
\begin{align*}
& \big\|\mathbf{1}_{(\g\setminus\tilde A^{(\ell)}(\mu^j,u^0))\cap\tilde A^{(\ell)}(\mu^j)} (F^{(\iota)}\psi)(\sqrt L) \chi(2^\ell U)\zeta_j(\mathbf U) f \big\|_1 \\
& \le \int_G \mathbf{1}_{(\g\setminus\tilde A^{(\ell)}(\mu^j,u^0))\cap\tilde A^{(\ell)}(\mu^j)}(x,u) \,  |f * \mathcal K_{\ell,j}^{(\iota)}(x,u)| \,d(x,u) \\
& \le \int_{\supp f} |f(x',u')| \int_{\bar A_\ell(\mu^j,u')} \big|\mathcal K_{\ell,j}^{(\iota)}\big((x',u')^{-1}(x,u)\big)\big|\,d(x,u)\,d(x',u')\\
& \le \|f\|_1 \sup_{(x',u')\in A^{(\ell)}(\mu^j,u^0)} \int_{\bar A_\ell(\mu^j,u')} \big|\mathcal K_{\ell,j}^{(\iota)}\big((x',u')^{-1}(x,u)\big)\big|\,d(x,u).
\end{align*}
If $(x',u')\in A^{(\ell)}(\mu^j,u^0)$ and $(x,u)\in \bar A_\ell(\mu^j,u')$, we have
\begin{equation}\label{eq:localization}
|x'| < C_1R,\quad |\bar P^{\mu^j}\! x'| < C_1 R_\ell,\quad |x|<3C_1R,\quad |\bar P^{\mu^j} \! x | < 6\kappa C_1 R_\ell R^\gamma.
\end{equation}
Note that
\[
(x',u')^{-1}(x,u)=(x-x',u-u'-\tfrac 1 2 [x',x]).
\]
Moreover, we have $|[y,y']\big|\lesssim \left|y\right|\left|y'\right|$ for all $y,y'\in\g_1$. Thus, writing
\[
x=P_0^{\mu^j}\! x + \bar P^{\mu^j}\! x\quad\text{and}\quad x'=P_0^{\mu^j}\! x' + \bar P^{\mu^j}\! x',
\]
and using \cref{eq:localization}, we see that there is some constant $C_0$ such that
\begin{equation}\label{eq:bracket}
\big|[x,x']\big| \le \big|[P_0^{\mu^j}\!x,P_0^{\mu^j}\!x']\big| + C_0 R_\ell R .
\end{equation}
Note that we have not use \cref{assumptionB} so far. Now, finally using \cref{assumptionB}, the first summand in \cref{eq:bracket} vanishes, and we have
\[
\big|[x,x']\big| \le C_0 R_\ell R .
\]
Thus, if we choose the constant $\bar C>0$ large enough,
\begin{align*}
| u-u'-\tfrac 1 2 [x',x]|
& \ge |u-u'| - \tfrac 1 2  C_0 R_\ell R\\
& \ge \bar C R_\ell R^{1+\gamma} - \tfrac 1 2  C_0 R_\ell R \gtrsim  R_\ell R^{1+\gamma},
\end{align*}
Hence, if $(x',u')\in A^{(\ell)}(\mu^j,u^0)$, then
\begin{align}
\int_{\bar A^{(\ell)}(\mu^j,u')} & \big|\mathcal K_{\ell,j}^{(\iota)}\big((x',u')^{-1}(x,u)\big)\big|\,d(x,u) \notag \\
& \lesssim (R_\ell R^{1+\gamma})^{-N} \int_{\tilde A^{(\ell)}(\mu^j)} | u-u'-\tfrac 1 2 [x',x]|^N \notag \\
& \hspace{3.7cm} \times \big|\mathcal K_{\ell,j}^{(\iota)}(x-x',u-u'-\tfrac 1 2 [x',x])\big|\,d(x,u)  \notag \\
& \le (R_\ell R^{1+\gamma})^{-N} | \tilde A^{(\ell)}(\mu^j)|^{1/2} \bigg(\int_G \big|| u|^N \mathcal K_{\ell,j}^{(\iota)}(x,u)\big|^2 \,d(x,u)\bigg)^{1/2}. \label{eq:apply-weighted-1}
\end{align}
Now the weighted Plancherel estimate of \cref{lem:Plancherel-2nd-layer} yields
\[
\bigg(\int_G \big||u|^N \mathcal K_{\ell,j}^{(\iota)}(x,u)\big|^2 \,d(x,u)\bigg)^{1/2} \lesssim_N R_\ell^{N-d_2/2} \|F^{(\iota)}\psi\|_{L^2_N}.
\]
Hence, the last line of \cref{eq:apply-weighted-1} is bounded by a constant times
\begin{equation}\label{eq:apply-weighted-2}
(R^{1+\gamma})^{-N} |\tilde A^{(\ell)}(\mu^j)|^{1/2} R_\ell^{-d_2/2} \|F^{(\iota)}\|_{L^2_N}.
\end{equation}
Since $|\tilde A^{(\ell)}(\mu^j)|\lesssim R^{d_1 + (2+\gamma)d_2}$ and $\smash{\|F^{(\iota)}\psi\|_{L^2_N}\lesssim R^N \|F^{(\iota)}\psi\|_2}$, \cref{eq:apply-weighted-2} can be bounded by a constant times
\[
R^{-\gamma N} R^{d_1 + (2+\gamma)d_2} R_\ell^{-d_2/2} \|F^{(\iota)}\psi\|_2.
\]
In particular, choosing $N=N(\gamma)$ large enough, we obtain
\begin{equation}\label{eq:interpolation-L1}
\big\| \,h_{\iota,\ell,j}\big|_{\g\setminus\tilde A^{(\ell)}(\mu^{j},u^0)} \big\|_1
 \lesssim_{\gamma,N} R^{-\gamma N} \|F^{(\iota)}\psi \|_2 \|f\|_1.
\end{equation}
Interpolation between the bounds \eqref{eq:interpolation-L1} and \eqref{eq:interpolation-L2} yields \cref{eq:rapid-decay-ii}.    
\end{proof}

\section{\texorpdfstring{Proof of Theorem \ref{thm:reduction}}{Proof of Theorem 4.1}}\label{sec:proof}

We first argue as in Step (1) to (4) of the proof of Proposition 7.1 in \cite{Ni25M}. First, using finite propagation speed and the left-invariance of $L$, it suffices to show that there is some $\varepsilon > 0$ such that
\begin{equation}\label{eq:reduction-1}
\big\| \textbf{1}_{B_{3R}^{d_{\mathrm{CC}}}} F^{(\iota)}(\sqrt L) f \big\|_p \lesssim 2^{-\varepsilon\iota} \|F^{(\iota)}\|_{L^2_s} \|f\|_p
\end{equation}
for all $\iota\in\N$ and all functions $f\in L^2(G)$ supported in the Carnot--Carathéodory ball $B_R^{d_{\mathrm{CC}}}$ centered at the origin, where $R=2^\iota$. 

Next, we put $\smash{\psi:=\sum_{|j|\le 2} \chi_j}$, where $(\chi_\iota)_{\iota \in\Z}$ is the dyadic decomposition of the beginning of \cref{sec:reduction}, and decompose $\smash{F^{(\iota)}}$ as
\[
F^{(\iota)}=F^{(\iota)}\psi + F^{(\iota)}(1-\psi).
\]
The second summand is an error term, which can be treated the Mikhlin--Hörmander type result of \cite{Ch91} and \cite{MaMe90}. Thus, instead of \cref{eq:reduction-1}, we have to show that there is some $\varepsilon > 0$ such that
\begin{equation}
\big\| \textbf{1}_{B_{3R}^{d_{\mathrm{CC}}}} (F^{(\iota)}\psi)(\sqrt L) f \big\|_p \lesssim 2^{-\varepsilon\iota} \|F^{(\iota)}\|_{L^2_s} \|f\|_p \label{eq:reduction-2}
\end{equation}
for all $\iota\in\N$  and for all $f\in L^2(G)$ supported in $B_R^{d_{\mathrm{CC}}}$, where $R=2^\iota$.

We further decompose $(F^{(\iota)}\psi)(\sqrt L)$. For $\ell\in\Z$, let $F_\ell^{(\iota)} : \R^2 \to \C$ be given by
\[
F_\ell^{(\iota)}(\lambda,\rho) = (F^{(\iota)}\psi) (\sqrt \lambda)\chi(2^\ell\rho)\quad\text{for }\lambda\geq 0
\]
and $F_\ell^{(\iota)}(\lambda,\rho)=0$ else. By \cref{lem:Plancherel}, there is a number $\ell_0\in\N$ depending only on the matrices $J_\mu$ and the inner product $\langle\cdot,\cdot\rangle$ on~$\g$ such that
\[
F_\ell^{(\iota)}(L,U) = 0 \quad \text{for all } \ell < -\ell_0,
\]
where $U = |\mathbf{U}|^{1/2}$ and $\mathbf{U} = (-iU_1, \dots, -iU_{d_2})$ with $U_1, \dots, U_{d_2}$ forming a basis of the second layer $\mathfrak{g}_2$.
 
Hence, we can decompose the function on the left-hand side of \eqref{eq:reduction-2} as
\begin{align*}
\mathbf{1}_{B_{3R}^{d_{\mathrm{CC}}}} (F^{(\iota)}\psi)(\sqrt L) f
& = \mathbf{1}_{B_{3R}^{d_{\mathrm{CC}}}} \bigg(\sum_{\ell =-\ell_0}^\iota + \sum_{\ell = \iota + 1}^\infty \bigg) F_\ell^{(\iota)}(L,U) f.
\end{align*}

Using the restriction type estimate of \cref{prop:p=ST} for the special case where $\delta$ is chosen so large that we only have one single cap $V_j$, we obtain
\begin{equation}\label{eq:restriction-type-ia}
\| F_\ell^{(\iota)}(L,U) \|_{p\to 2}
\lesssim 2^{-\ell d_2(\frac 1p - \frac 1 2)} \| (F^{(\iota)}\psi)(\sqrt{\cdot}\,)\|_{2^{\ell},2}.
\end{equation}
By the Sobolev type embedding of \cref{eq:sobolev-embedding}, since $\supp \psi\subseteq \R^+$, we get
\begin{equation}\label{eq:sobolev-type-ii}
\| (F^{(\iota)}\psi)(\sqrt{\cdot}\,)\|_{2^{\ell},2} \lesssim_{\tilde s} \| F^{(\iota)} \|_2 + 2^{-\ell \tilde s} \| F^{(\iota)} \|_{L^2_{\tilde s}} \quad \text{for every } \tilde s>1/2.
\end{equation}
Recall that $|B_R^{d_{\mathrm{CC}}}| \sim R^{d_1+2d_2}$ by \eqref{eq:cc-volume}. Using Hölder's inequality (with $1/q=1/p-1/2$) in combination with \cref{eq:restriction-type-ia} and \cref{eq:sobolev-type-ii}, we get
\begin{align*}
& \bigg \|  \mathbf{1}_{B_{3R}^{d_{\mathrm{CC}}}} \sum_{\ell = \iota + 1}^\infty F_\ell^{(\iota)}(L,U) f \bigg \|_p
   \lesssim R^{(d_1+2d_2)/q} \sum_{\ell = \iota + 1}^\infty \| F_\ell^{(\iota)}(L,U) f \|_2 \\
 & \lesssim 2^{\iota (d_1+2d_2)/q}\sum_{\ell = \iota + 1}^\infty 2^{-\ell d_2/q} \big ( \| F^{(\iota)} \|_2 + 2^{-\ell \tilde s} \| F^{(\iota)} \|_{L^2_{\tilde s}} \big) \|f\|_p\\
 & \lesssim 2^{\iota d/q}  \big( \|F^{(\iota)}\|_2 + 2^{-\iota \tilde s} \|F^{(\iota)}\|_{L^2_{\tilde s}} \big) \|f\|_p
  \sim \|F^{(\iota)}\|_{L^2_{d/q}} \|f\|_p.
\end{align*}
The last term is comparable to $2^{-\varepsilon \iota} \|F^{(\iota)}\|_{L^2_{s}}\|f\|_p$ if we choose $\varepsilon\in (0,s-d/q)$.

Altogether, we are left showing that there is some $\varepsilon>0$ such that
\begin{equation}\label{eq:small-ell}
\bigg\| \mathbf{1}_{B_{3R}^{d_{\mathrm{CC}}}} \sum_{\ell =-\ell_0}^\iota  F_\ell^{(\iota)}(L,U) f\bigg\|_p
 \lesssim  2^{-\varepsilon\iota} \|F^{(\iota)}\|_{L^2_s} \|f\|_p
\end{equation}
for all $\iota\in\N$  and for all $f\in L^2(G)$ supported in $B_R^{d_{\mathrm{CC}}}$.

We further refine the decomposition of $F_\ell^{(\iota)}(L,U)$. Let $\zeta_j$ be as in \cref{sec:restriction-caps}, that is, $\zeta_j|_{S^{d_2-1}}$ is supported on a cap of size $\delta$, and let again $\mu^j\in \supp \zeta_j$. As in \cref{prop:rapid-decay-1st-layer}, we choose $\delta=R/R_\ell$. 
Let $\smash{F_{\ell,j}^{(\iota)} : \R\times \R^2 \to \C}$ be given by
\[
F_{\ell,j}^{(\iota)}(\lambda,\mu) = F_\ell^{(\iota)}(\lambda,|\mu|)\,\zeta_j(\mu)\quad\text{for }\lambda\in\R,\mu\in\R^{d_2}.
\]
Then
\[
F_\ell^{(\iota)}(L,U) = \sum_{j\in J} F_{\ell,j}^{(\iota)}(L,\mathbf U),
\]
where $J\sim \delta^{-(d_2-1)}$.
We use the bound
\begin{equation}\label{eq:triangle-caps}
\bigg\| \mathbf{1}_{B_{3R}^{d_{\mathrm{CC}}}} \sum_{\ell =-\ell_0}^\iota  F_\ell^{(\iota)}(L,U) f\bigg\|_p
\le \sum_{\ell =-\ell_0}^\iota \sum_{j\in J} \big\| \mathbf{1}_{B_{3R}^{d_{\mathrm{CC}}}}   F_{\ell,j}^{(\iota)}(L,\mathbf U) f \big\|_p,
\end{equation}
where each summand will be treated individually. 

\smallskip

\subsection{Localization on the first layer}

Let $C>0$ be as in \cref{eq:C}, that is,
\[
B_R^{d_{\mathrm{CC}}} \subseteq B_{CR} \times B_{CR^2}.
\]
For all $\ell \in\{-\ell_0,\dots,\iota\}$ and $j\in J$, there are $x_{m,j}^{(\ell)}\in \g_1$ and disjoint sets $B_m^{(\ell)}(\mu^{j})$, $m \in\{1,\dots,M_\ell\}$, such that
\[
|x_{m}^{(\ell)}|<CR \quad\text{and}  \quad |P^{\mu^j}\!(x_{m,j}^{(\ell)} - x_{m',j}^{(\ell)})| > R_\ell/2 \quad\text{for }  m\neq m',
\]
and $B_m^{(\ell)}(\mu^{j})$ is contained in the ball
\[
B^{(\ell)}(x_{m,j}^{(\ell)},\mu^{j}) = \big\{ (x,u) \in B_{CR}\times B_{CR^2} :  |\bar P^{\mu^j} \! (x-x_{m,j}^{(\ell)}) | < CR_\ell \big\}.
\]
Then, for any $\gamma>0$, the number of overlapping balls
\[
\tilde B_m^{(\ell)}(x_{m,j}^{(\ell)},\mu^{j}) =  \big\{ (x,u) \in B_{3CR}\times B_{9CR^2} :  |\bar P^{\mu^j} \! (x-x^0) | < 6\kappa C R_\ell R^\gamma\big\}.
\]
is bounded by a constant $N_\gamma \lesssim_\iota 1$ (which is independent of $\ell$), where the notation
\begin{equation}\label{eq:iota-notation}
A\lesssim_\iota B
\end{equation}
means that $A\le R^{C(p,d_1,d_2)\gamma} B$ for some constant $C(p,d_1,d_2)>0$ depending only on the parameters $p,d_1,d_2$.

Using this decomposition, we write the left-hand side of \cref{eq:small-ell} as
\[
\mathbf{1}_{B_{3R}^{d_{\mathrm{CC}}}} F_{\ell,j}^{(\iota)}(L,\mathbf U) f
= \sum_{m=1}^{M_{\ell}} \Big(\mathbf{1}_{\tilde B_{m}^{(\ell)}(\mu^j)}+\mathbf{1}_{\g\setminus\tilde B_{m}^{(\ell)}(\mu^j)}\Big) \mathbf{1}_{B_{3R}^{d_{\mathrm{CC}}}} F_{\ell,j}^{(\iota)}(L,\mathbf U) \big(f|_{B_{m}^{(\ell)}(\mu^j)}\big).
\]
Since $M_\ell\sim (R/R_\ell)^{d_1}\lesssim R^{d_1}$, the second summands are negligible thanks to the rapid decay from \cref{prop:rapid-decay-1st-layer}, where we use the Sobolev embedding $L^\infty \subseteq L^2_\alpha$, $\alpha>1/2$. Indeed, we obtain
\[
\sum_{\ell =-\ell_0}^\iota \sum_{j\in J} \sum_{m=1}^{M_\ell} \big\| \mathbf{1}_{B_{3R}^{d_{\mathrm{CC}}}\cap (\g\setminus\tilde B_{m}^{(\ell)}(\mu^j))}   F_{\ell,j}^{(\iota)}(L,\mathbf U) f \big\|_p \lesssim_{\gamma,N}  R^{-\gamma N} \|F^{(\iota)}\|_{2} \|f\|_p.
\]
The parameter $\gamma>0$ will be chosen later in the proof.

Hence, in place of \cref{eq:triangle-caps}, it suffices to show
\begin{equation}\label{eq:small-ell-ii}
\sum_{\ell=-\ell_0}^\iota \sum_{j\in J} 
\big\| \sum_{m=1}^{M_{\ell}}  \mathbf{1}_{\tilde B_m^{(\ell)}(\mu^j)\cap B_{3R}^{d_{\mathrm{CC}}}} F_{\ell,j}^{(\iota)}(L,\mathbf U) \big(f|_{B_{m}^{(\ell)}(\mu^j)}\big) \big\|_p
 \lesssim  2^{-\varepsilon\iota} \|F^{(\iota)}\|_{L^2_s} \|f\|_p.
\end{equation}
On the other hand, to prove \cref{eq:small-ell-ii}, it suffices to show
\begin{equation}\label{eq:main-term-c}
\big\| \mathbf{1}_{\tilde B_m^{(\ell)}(\mu^j)\cap B_{3R}^{d_{\mathrm{CC}}}} F_{\ell,j}^{(\iota)}(L,\mathbf U) \big(f|_{B_{m}^{(\ell)}(\mu^j)}\big) \big\|_p
 \lesssim J^{-1} 2^{-\varepsilon\iota} \|F^{(\iota)}\|_{L^2_s} \,\big\|f|_{B_{m}^{(\ell)}(\mu^j)}\big\|_p.
\end{equation}
Indeed, thanks to the bounded overlap of the balls $\tilde B_{m}^{(\ell)}(\mu^j)$, given \cref{eq:main-term-c}, we obtain
\begin{align*}
& \sum_{\ell=-\ell_0}^\iota \sum_{j\in J} 
\big\| \sum_{m=1}^{M_{\ell}}  \mathbf{1}_{\tilde B_m^{(\ell)}(\mu^j)\cap B_{3R}^{d_{\mathrm{CC}}}} F_{\ell,j}^{(\iota)}(L,\mathbf U) \big(f|_{B_{m}^{(\ell)}(\mu^j)}\big) \big\|_p \\
& \lesssim \sum_{\ell=-\ell_0}^\iota \sum_{j\in J} \bigg( \sum_{m=1}^{M_{\ell}} \big\| \mathbf{1}_{\tilde B_m^{(\ell)}(\mu^j)\cap B_{3R}^{d_{\mathrm{CC}}}} F_{\ell,j}^{(\iota)}(L,\mathbf U) \big(f|_{B_{m}^{(\ell)}(\mu^j)}\big) \big\|_p^p \bigg)^{\frac 1 p} \\
& \lesssim \sum_{\ell=-\ell_0}^\iota \sum_{j\in J}J^{-1}  2^{-\varepsilon\iota}\|F^{(\iota)}\|_{L^2_s}  \bigg( \sum_{m=1}^{M_{\ell}} \big\|  f|_{B_{m}^{(\ell)}(\mu^j)} \big\|_p^p \bigg)^{\frac 1 p} \\
& \lesssim 2^{-\tilde \varepsilon\iota} \|F^{(\iota)}\|_{L^2_s} \|f\|_p \qquad \text{for } 0<\tilde\varepsilon<\varepsilon.
\end{align*}

\subsection{Translation back to the origin}

To show \cref{eq:small-ell}, we use translation invariance to translate the balls $B_{m}^{(\ell)}(\mu^j)$ and $\tilde B_{m}^{(\ell)}(\mu^j)$ to the origin.

Let $\lambda_{(x_{m,j}^{(\ell)},0)}$ denote the left multiplication by $(x_{m,j}^{(\ell)},0)$ on $G$. Note that
\begin{equation}\label{eq:xu-translated}
\lambda_{(x_{m,j}^{(\ell)},0)}(x,u)\in B_{m}^{(\ell)}(\mu^j)
\end{equation}
implies $|\bar P^{\mu^j} \! x | < CR_\ell$ and $\big|u+\tfrac 1 2 [x_{m,j}^{(\ell)},x]\big| \le C R^2$, and thus in particular
\[
|u| \lesssim R^2 + |x_{m,j}^{(\ell)}|\left|x\right| \lesssim R^2.
\]
Thus, there is a constant $C_2>0$ such that \cref{eq:xu-translated} implies $(x,u)\in A^{(\ell)}(\mu^j)$, where
\[
A^{(\ell)}(\mu^{j}) := \big\{ (x,u) \in B_{C R}\times B_{C_2 R^2} :  |\bar P^{\mu^j} \! x | < C R_\ell \big\}.
\]
Similarly, if
\[
\lambda_{(x_{m,j}^{(\ell)},0)}(x,u)\in \tilde B_{m}^{(\ell)}(\mu^j)
\]
then $(x,u)\in \tilde A^{(\ell)}(\mu^{j})$, where
\[
\tilde A^{(\ell)}(\mu^{j}) := \big\{ (x,u) \in B_{3C R}\times B_{9C_2 R^{2+\gamma}} : |\bar P^{\mu^j} \! x | < 6\kappa C R_\ell R^\gamma \big\},
\]
provided that $C_2>0$ is chosen sufficiently large. Since $F_{\ell,j}^{(\iota)}(L,\mathbf U)$ is left-invariant,
\[
\big(F_{\ell,j}^{(\iota)}(L,\mathbf U) \big(f|_{B_{m}^{(\ell)}(\mu^j)}\big)\big)(\lambda_{(x_{m,j}^{(\ell)},0)}(x,u)) = 
F_{\ell,j}^{(\iota)}(L,\mathbf U) \tilde f_{m,j}^{(\ell)}(x,u),
\]
where $\tilde f_{m,j}^{(\ell)}(x,u) := f_{m,j}^{(\ell)}(\lambda_{(x_{m,j}^{(\ell)},0)}(x,u))$ is supported in $A^{(\ell)}(\mu^j)$. Note that
\[
\big\|\mathbf{1}_{\tilde B_{m}^{(\ell)}(\mu^j)} F_{\ell,j}^{(\iota)}(L,\mathbf U) \big(f|_{B_{m}^{(\ell)}(\mu^j)}\big) \big\|_p
 = \|\mathbf{1}_{\tilde A^{(\ell)}(\mu^j)} F_{\ell,j}^{(\iota)}(L,\mathbf U) \tilde f_{m,j}^{(\ell)} \|_p.
\]
Thus, in place of \cref{eq:main-term-c}, it suffices to show that there is some $\varepsilon>0$ such that
\begin{equation}\label{eq:main-term-d}
\|\mathbf{1}_{\tilde A^{(\ell)}(\mu^j)} F_{\ell,j}^{(\iota)}(L,\mathbf U) f\|_p \lesssim  J^{-1} 2^{-\varepsilon\iota} \|F^{(\iota)}\|_{L^2_s}\|f\|_p
\end{equation}
for all functions $f$ supported in $A^{(\ell)}(\mu^j)$. 

\subsection{Localization on the second layer}

We decompose $A^{(\ell)}(\mu^j)$ as
\[
A^{(\ell)}(\mu^j) = \bigcup_{k=1}^{K_\ell} A^{(\ell)}_k(\mu^j),
\]
where $A^{(\ell)}_k(\mu^j)\subseteq A^{(\ell)}(\mu^j)$ are disjoint subsets such that
\[
A^{(\ell)}_k(\mu^j) \subseteq \big\{ (x,u) \in B_{C R}\times B_{C_2 R^2} :  |\bar P^{\mu^j} \! x | < C R_\ell \text{ and } |u-u^{(\ell)}_k|< C_2 R_\ell R \big\},
\]
and $u^{(\ell)}_k\in A^{(\ell)}(\mu^j)$ satify $|u^{(\ell)}_k - u^{(\ell)}_{k'}| > R_\ell R/2$ for $k\neq k'$. Moreover, let
\[
\tilde A^{(\ell)}_k(\mu^j) :=\{ (x,u) \in \tilde A^{(\ell)}(\mu^j) : |u- u^{(\ell)}_k| < 6\kappa C R_\ell R^{1+\gamma} \}.
\]
Then the maximal number $N_\gamma$ of overlapping sets $\tilde A^{(\ell)}_k(\mu^j)$ is bounded by $N_\gamma \lesssim_\iota 1$. Now, if $f$ is supported in $A^{(\ell)}(\mu^j)$, we decompose it as
\[
f = \sum_{k=1}^{K_\ell} f|_{A^{(\ell)}_k(\mu^j)}.
\]
We decompose the function $h_{k,j}^{(\ell)}:=\mathbf{1}_{\tilde A^{(\ell)}(\mu^j)} F_{\ell,j}^{(\iota)}(L,\mathbf U) ( f|_{A^{(\ell)}_k(\mu^j)})$ as
\[
h_{k,j}^{(\ell)} = \mathbf{1}_{\tilde A^{(\ell)}_k(\mu^j)} h_{k,j}^{(\ell)} + \mathbf{1}_{\g\setminus\tilde A^{(\ell)}_k(\mu^j)} h_{k,j}^{(\ell)}.
\]
Since $K_\ell\lesssim (R/R_\ell)^{d_2}\lesssim R^{d_2}$, \cref{prop:rapid-decay-2nd-layer} yields
\[
\bigg\| \sum_{k=1}^{K_{\ell}} \mathbf{1}_{\g\setminus\tilde A^{(\ell)}_k(\mu^j)} h_{k,j}^{(\ell)} \bigg\|_p
 \lesssim_{\gamma,N} R^{-\gamma N} \|F^{(\iota)}\|_2 \|f\|_p,\quad N>0.
\]
On the other hand, since the balls $\tilde A^{(\ell)}_k(\mu^j)$ have bounded overlap, we obtain
\[
\bigg\| \sum_{k=1}^{K_{\ell}}  \mathbf{1}_{\tilde A^{(\ell)}_k(\mu^j)} h_{k,j}^{(\ell)} \bigg\|_p
\lesssim_\iota 
\bigg( \sum_{k=1}^{K_{\ell}} \| \mathbf{1}_{\tilde A^{(\ell)}_k(\mu^j)} h_{k,j}^{(\ell)} \|_p^p \bigg)^{1/p}.
\]
Moreover, note that
\[
\bigg( \sum_{k=1}^{K_{\ell}} \|f|_{A^{(\ell)}_k(\mu^j)}\|_p^p \bigg)^{1/p}
= \|f\|_p.
\]
Thus, to prove \cref{eq:main-term-d}, it suffices to show that there is some $\varepsilon>0$ such that
\begin{equation}\label{eq:main-term-e}
\big\| \mathbf{1}_{\tilde A^{(\ell)}_k(\mu^j)} F_{\ell,j}^{(\iota)}(L,\mathbf U) f \big\|_p \lesssim J^{-1} 2^{-\varepsilon\iota} \|F^{(\iota)}\|_{L^2_s} \|f\|_p
\end{equation}
for all functions $f$ supported in $A^{(\ell)}_k(\mu^j)$.

\subsection{Applying the restriction type estimate}

We apply the restriction type estimate of \cref{thm:restriction-type} for caps of size $\delta = R/R_\ell$. Recall that our assumptions imply in particular $d_1\ge d_2$, whence
\[
\ST{d_1,d_2} = \min \{\ST{d_1},\ST{d_2}\} = \ST{d_2}.
\]
Thus, \cref{thm:restriction-type} yields
\begin{multline}\label{eq:restriction-iii}
\norm{ F_{\ell,j}^{(\iota)}(L,\mathbf U) }_{p\to 2}
\lesssim R_\ell^{-d_2/q} \Big(\frac{R_\ell}{R}\Big)^{\frac{d_2-1} 2 (1-\theta_p)} \|(F^{(\iota)}\psi)(\sqrt{\cdot}\,)\|_2^{1-\theta_p} \\ \times \|(F^{(\iota)}\psi)(\sqrt{\cdot}\,)\|_{2^{\ell},2}^{\theta_p},
\end{multline}
where $\theta_p\in [0,1]$ satisfies $1/p = (1-\theta_p) + \theta_p/\ST{d_2}$. 
Using the same Sobolev type embedding as in \cref{eq:sobolev-type-ii}, we can bound the right-hand side of \cref{eq:restriction-iii} by a constant times
\[
R_\ell^{-d_2/q} \Big(\frac {R_\ell} R \Big)^{\frac{d_2-1} 2 (1-\theta_p) - \theta_p \tilde s} \|F^{(\iota)}\|_2 \quad \text{for all } \tilde s >1/2.
\]
Let $J_\mu$ be again be the matrix of \cref{eq:skew-form-ii}. Recall that $\mu \mapsto \dim \ker J_\mu$ is constant on $\g_2^*\setminus\{0\}$ by \cref{assumptionA}. Let $r_0:=\dim \ker J_\mu$ for $\mu\neq 0$ and $\bar d_1:=d_1-r_0$. Then
\[
|\tilde A^{(\ell)}_k(\mu^j)|\lesssim R^{r_0} R_\ell^{\bar d_1} (R_\ell R)^{d_2}
\]
Let again $q\in [2,\infty)$ be such that $1/q=1/p-1/2$. Using Hölder's inequality in combination with the restriction type estimate \cref{eq:restriction-iii}, we get
\begin{align}
& \big\| \mathbf{1}_{\tilde A^{(\ell)}_k(\mu^j)} F_{\ell,j}^{(\iota)}(L,\mathbf U) f \big\|_p
 \lesssim_\iota \big( R^{r_0} R_\ell^{\bar d_1} (R_\ell R)^{d_2} \big)^{1/q} \| F_{\ell,j}^{(\iota)}(L,\mathbf U) f\|_2 \notag \\
 & \lesssim \big( R^{r_0} R_\ell^{\bar d_1} (R_\ell R)^{d_2} \big)^{1/q} R_\ell^{-d_2/q} \Big(\frac {R_\ell} R \Big)^{\frac{d_2-1} 2 (1-\theta_p) - \theta_p \tilde s} \|F^{(\iota)}\|_2 \|f\|_p. \label{eq:after-H}
\end{align}
Using
\[
\|F^{(\iota)}\|_2\sim  R^{-d/q} \|F^{(\iota)}\|_{L^2_{d/q}} =  R^{-(r_0+\bar d_1+d_2)/q} \|F^{(\iota)}\|_{L^2_{d/q}},
\]
we see that \cref{eq:after-H} is comparable to
\[
\Big(\frac {R_\ell} R \Big)^{\bar d_1/q + \frac{d_2-1} 2 (1-\theta_p) - \theta_p \tilde s} \|F^{(\iota)}\|_{L^2_{d/q}} \|f\|_p.
\]
Recall that the number of caps in our decomposition is
\[
J \sim \delta^{-(d_2-1)} = \Big(\frac {R_\ell} R \Big)^{-(d_2-1)}.
\]
This implies
\begin{multline}
J\, \| \mathbf{1}_{\tilde A^{(\ell)}_k(\mu^j)} F_{\ell,j}^{(\iota)}(L,\mathbf U) f \big\|_p \lesssim
\Big(\frac {R_\ell} R \Big)^{\bar d_1/q + \frac{d_2-1} 2 (1-\theta_p) - \theta_p \tilde s - (d_2-1)}  \\ \times \|F^{(\iota)}\|_{L^2_{d/q}} \|f\|_p. \label{eq:exponent}
\end{multline}
To establish \cref{eq:main-term-e} and complete the proof of \cref{thm:reduction}, it suffices to ensure that
\begin{equation}\label{eq:final-condition}
\frac {\bar d_1} q - \theta_p \frac{d_2} 2 - \frac{d_2-1} 2 \ge 0.
\end{equation}
Indeed, suppose that \cref{eq:final-condition} is satisfied. Note that
\[
\|F^{(\iota)}\|_{L^2_{d/q}}\sim R^{d/q-s}\|F^{(\iota)}\|_{L^2_{s}}
\]
and that we have assumed $s>d/q$ in \cref{thm:reduction}. Thus, choosing $\gamma>0$ small enough and $\tilde s > 1/2$ sufficiently close to $1/2$, we can ensure that the exponent in \cref{eq:exponent} is positive.

Using the identities
\[
\frac{1}{q} = \frac 1 2 - \frac{1}{p'} \quad \text{and} \quad
\theta_p = \frac{\ST{d_2}'}{p'},
\]
we see that \cref{eq:final-condition} is equivalent to
\[
\frac{1}{p'} \le \frac{\bar d_1-(d_2-1)}{2\bar d_1 + d_2\,\ST{d_2}'}.
\]
Since
\[
\ST{d_2}'=\frac{2(d_2+1)}{d_2-1},
\]
this in turn is equivalent to
\[
p
 \le \frac{2\bar d_1 +d_2\,\ST{d_2}'}{\bar d_1 + d_2\,\ST{d_2}' +(d_2-1)} \\
 = 
\frac{2\bar d_1(d_2-1) +2d_2(d_2+1)}{\bar d_1(d_2-1) + 3d_2^2+1}.
\]
This finishes the proof of \cref{thm:reduction}.\hfill$\square$


\end{document}